\documentclass[10pt,reqno]{amsart}
\usepackage[margin=0.8in]{geometry}
\usepackage{amsthm, amsmath,amsfonts,amssymb,euscript,hyperref,graphics,color,slashed}
\usepackage{graphicx}
\usepackage{mathrsfs}
\usepackage{comment}
\usepackage{latexsym}
\usepackage[makeroom]{cancel}

\def\inte#1{
\displaystyle\mathop{#1\kern0pt}^\circ }

\let\e=\varepsilon
\let\z=\zeta

\let\f=\frac

\let\p=\psi

\let\D=\Delta

\let\wt=\widetilde

\def\virgp{\raise 2pt\hbox{,}}
\def\cdotpv{\raise 2pt\hbox{;}}

\def\eqdefa{\buildrel\hbox{\footnotesize def}\over =}

\def\C{\mathop{\mathbb C\kern 0pt}\nolimits}
\def\DD{\mathop{\mathbb D\kern 0pt}\nolimits}
\def\EE{\mathop{{\mathbb E \kern 0pt}}\nolimits}
\def\K{\mathop{\mathbb K\kern 0pt}\nolimits}
\def\N{\mathop{\mathbb N\kern 0pt}\nolimits}
\def\Q{\mathop{\mathbb Q\kern 0pt}\nolimits}
\def\R{\mathop{\mathbb R\kern 0pt}\nolimits}
\def\SS{\mathop{\mathbb S\kern 0pt}\nolimits}
\def\ZZ{\mathop{\mathbb Z\kern 0pt}\nolimits}
\def\TT{\mathop{\mathbb T\kern 0pt}\nolimits}
\def\P{\mathop{\mathbb P\kern 0pt}\nolimits}

\newcommand{\Z}{{\ZZ}}

\def\curl{\mathop{\rm curl}\nolimits}

\def\na{\nabla}
\def\p{\partial}

\newcommand{\beq}{\begin{equation}}
\newcommand{\eeq}{\end{equation}}
\newcommand{\ben}{\begin{eqnarray}}
\newcommand{\een}{\end{eqnarray}}
\newcommand{\beno}{\begin{eqnarray*}}
\newcommand{\eeno}{\end{eqnarray*}}

\newtheorem{thm}{Theorem}[section]

\newcommand{\vv}[1]{\boldsymbol{#1}}

\def\div{\text{div}\,}
\def\curl{\text{curl}\,}

\newtheorem*{Main Theorem}{Main Theorem}
\newtheorem{theorem}{Theorem}[section]
\newtheorem{lemma}[theorem]{Lemma}
\newtheorem{proposition}[theorem]{Proposition}

\newtheorem{remark}[theorem]{Remark}

\numberwithin{equation}{section}

\begin{document}
\title[Well-posedness]{Long time existence for a two-dimensional strongly dispersive Boussinesq system}

\author{Jean-Claude Saut}
\address{Laboratoire de Math\' ematiques, UMR 8628\\
Universit\' e Paris-Saclay, Paris-Sud et CNRS\\ 91405 Orsay, France}
\email{jean-claude.saut@universite-paris-saclay.fr}

\author[Li XU]{Li Xu}
\address{School of Mathematical Sciences, Beihang University\\  100191 Beijing, China}
\email{xuliice@buaa.edu.cn}

\date{}
\maketitle

\vspace{1cm}
 \textit{Abstract}. We prove a long time existence result for the solutions of a two-dimensional Boussinesq system modeling the propagation of long, weakly nonlinear water waves. This system is exceptional in the sense that it is the only
 linearly well-posed system in the (abcd) family of Boussinesq systems whose eigenvalues of the linearized system have nontrivial zeroes. This new difficulty is solved by the use of "good unknowns " and of normal form techniques.


\tableofcontents

\setcounter{equation}{0}
\section{Introduction}
\subsection{The general setting}

The four-parameter (abcd) Boussinesq systems for {\it long wavelength, small amplitude} gravity-capillary  surface water waves introduced in \cite{BCL, BCS1} couples  the elevation of the wave $\zeta=\zeta(x,t)$ to a measure of the horizontal velocity $\vv v=\vv v(x,t), x\in \R^N, N=1,2, t\in \R$ and read as follows:
\beq\label{Bsq 1}\left\{\begin{aligned}
&\p_t\z+\na\cdot\vv v+\epsilon\na\cdot(\z\vv v)+\epsilon\bigl(a\na\cdot\Delta\vv v-b\Delta\p_t\z\bigr)=0,\\
&\p_t\vv v+\na\z+\f{\epsilon}{2}\na(|\vv v|^2)+\epsilon\bigl(c\na\Delta\z-d\Delta\p_t\vv v\bigr)=\vv 0.
\end{aligned}\right.\eeq

Here $a, b, c, d$ are modeling parameters which satisfy the
constraint $a+b+c+d=\frac{1}{3}-\tau $  where $\tau\geq 0$ is a measure of surface tension effects, $\tau=0$ for pure gravity waves.

In \eqref{Bsq 1},  the small parameter $\epsilon$ is defined by
$$\epsilon=a/h\sim (h/\lambda)^2,$$
 where $h$ denotes the mean depth of the fluid, $a$ a typical amplitude of the wave and $\lambda$ a typical horizontal wavelength.

It was established in \cite{BCL} that, in suitable Sobolev classes, the error between the  solutions of the full water waves system and their approximation given by \eqref{Bsq 1} is of order $O(\epsilon^2 t).$ Since the corresponding solutions of the full water wave system have been proven in \cite{AL, La1} to exist on time scales of order $O(1/\epsilon)$, one needs to establish {\it long time existence results} for the Boussinesq systems, the "optimal" existence time scale being $O(1/\epsilon)$. Note that the "dispersive" methods used to prove the local well-posedness of the corresponding Cauchy problems in low order Sobolev spaces lead to time scales of order $O(1/\sqrt \epsilon)$ (see for instance \cite{LPS}).

 The existence of solutions of the Boussinesq systems on time scales of order $O(1/\epsilon)$ has been established  in \cite{ Bu, Bu2, MSZ, SX, SWX} for all the locally-well posed Boussinesq systems except the case $b=d=0, a=c>0$ which is in some sense special since the "generic" case  $b=d=0, a, c>0, a\neq c$ is linearly ill-posed. We also refer to \cite{SX3} for the case of Full-Dispersion Boussinesq systems.

 \begin{remark}
 The global well-posedness of Boussinesq systems has been only established in a few cases, including the one-dimensional  case $a=c=b=0, d>0$ that can be viewed as a dispersive perturbation of the hyperbolic Saint-Venant (shallow water) system, see \cite{A, Sc, MTZ}, and the  Hamiltonian cases $b=d>0, a\leq 0, c<0$, see \cite{BCS2} for the one-dimensional case and \cite{Hu} for the two-dimensional case. We also refer to \cite{KMPP, KM} for scattering results in the energy space for those one-dimensional Hamiltonian systems  when $b=d>0.$
\end{remark}

\vspace{0.3cm}
Recall that the linearization of \eqref{Bsq 1} around the null solution is well-posed (see \cite{BCS1}) provided that
\beq\label{Bsq 2}
a\leq0,\quad c\leq0,\quad b\geq0,\quad d\geq0,
\eeq
\beq\label{Bsq 3}
\text{or}\quad a=c>0,\quad b\geq0,\quad d\geq0.
\eeq

Actually the linear well-posedness occurs when the non zero eigenvalues of the linearization of \eqref{Bsq 1} at $(0, 0)$

$$ \lambda_{\pm}(\xi)=\pm i |\xi|\left(\frac{(1-\epsilon a|\xi|^2)(1-\epsilon c|\xi|^2)}{(1+\epsilon d|\xi|^2)(1+\epsilon b|\xi|^2)}\right)^{\frac{1}{2}}.$$
are purely imaginary.

\vspace{0.3cm}
This paper will focus on  the exceptional case \eqref{Bsq 3} with $b=d=0,\, a=c=1$ which is the only linearly well-posed case with eigenvalues having non trivial zeroes, leading to difficulties not present in the other well-posed systems.

The one-dimensional (1D) case was considered in \cite{SX2} and we will restrict to the two-dimensional(2D) still open case under the physical condition $\curl\vv v=0$, $N=2.$


\vspace{0.3cm}
In this paper, we shall  establish the long time existence theory for the following strongly dispersive (2D) Boussinesq system
\beq\label{Bsq2D}\left\{\begin{aligned}
&\p_t\z+(1+\epsilon\D)\na\cdot\vv v+\epsilon\na\cdot(\z\vv v)=0,\\
&\p_t\vv v+(1+\epsilon\D)\na\z+\f{\epsilon}{2}\na\bigl(|\vv v|^2\bigr)=\vv 0,
\end{aligned}\right.\eeq
with the initial data
\beq\label{initial}
\z|_{t=0}=\z_0,\quad \vv v|_{t=0}=\vv v_0.
\eeq
We shall assume that $\vv v$ is curl-free, i.e.,
\beq\label{curl free condition}
\p_1v^2=\p_2v^1.
\eeq
If \eqref{curl free condition} holds for $t=0$, then it holds for all time $t>0$ in the lifespan of the solutions to the system \eqref{Bsq2D}.

It was established in \cite{LPS}, using various dispersive properties of the underlying linear group (see \cite{KPV1, KPV2, KPV3} in the context of the KdV equation) that the Cauchy problem for \eqref{Bsq2D} is locally well-posed for initial data in $H^s(\R^2)\times H^s(\R^2)^2,\; s>3/2$. However, the corresponding lifespan of the solution is $O(1/\sqrt \epsilon),$ smaller though than the expected  $O(1/\epsilon).$  The fact that purely dispersive methods do not yield the correct lifespan is understandable since  this kind of methods are stable by perturbations which destroy the nonlinear  structure of the system (essentially that of the shallow-water system) which of course plays a crucial role in the long time behavior of the solution.

In \cite{SX2}, by introducing good unknowns(in the sense of Alinhac in \cite{ABZ,AM,Alin}), the authors symmetrized the 1D version of \eqref{Bsq2D} to avoid the loss of derivatives. Using  normal formal techniques on the set away from the spatial resonance set, the authors established the well-posedness of the 1D Boussinesq over the time scalar $\epsilon^{-\f23}$.

The goal of the present paper is to extend the lifespan of the local solution for the 2D Boussinesq \eqref{Bsq2D}.

\subsection{The main result}

We now state the main result of this paper as follows
\begin{thm}\label{main theorem}
Assume that  $(\z_0,\vv v_0)\in H^{N_0}(\R^2)$ for some $N_0\geq 5$ satisfying $\widehat{\z_0}(0)=0,\,\,\widehat{\vv v_0}(0)=\vv 0$ and
\beq\label{initial assumption}
\p_1v_0^2=\p_2v_0^1.
\eeq
 Then there exist a small  $\epsilon_0>0$ and a constant $T_0=T_0(\|\z_0\|_{H^{N_0}}+\|\vv v_0\|_{H^{N_0}})$ such that for any $\epsilon\in(0,\epsilon_0]$, there exists a unique solution $(\z,\vv v)\in C(0,T_0\epsilon^{-\f23};H^{N_0}(\R^2))$ of system \eqref{Bsq2D}-\eqref{initial} such that
\beq\label{energy estimate}
\sup_{t\in[0,T_0\epsilon^{-\f23}]}\bigl(\|\z(t)\|_{H^{N_0}}+\|\vv v(t)\|_{H^{N_0}}\bigr)\leq C\bigl(\|\z_0\|_{H^{N_0}}+\|\vv v_0\|_{H^{N_0}}\bigr),
\eeq
where $C>0$ is a universal constant.
\end{thm}
\begin{remark}
If $\widehat{\z_0}(0)=0,\,\,\widehat{\vv v_0}(0)=0$, \eqref{Bsq2D} shows that $\widehat{\z}(t,0)=0,\,\,\widehat{\vv v}(t,0)=0$ holds for all time $t>0$ in the lifespan of the solutions to \eqref{Bsq2D}. Therefore, throughout the whole paper, we shall use the condition   $\widehat{\z}(t,0)=0,\,\,\widehat{\vv v}(t,0)=0$ so that we could use the homogenous Littlewood-Paley decompositions.
\end{remark}

\begin{remark}
The curl-free condition \eqref{curl free condition} of the velocity $\vv v$  guarantees that the system \eqref{Bsq2D} could be symmetrized so that there is no loss of derivatives. With \eqref{curl free condition}, the principal part of \eqref{Bsq2D} is similar to the 1D case of \eqref{Bsq2D} in \cite{SX2}.
\end{remark}

\begin{remark}
Reaching the expected time scale $O(1/\epsilon)$ for the solutions of \eqref{Bsq2D} is still an open problem.
\end{remark}

\subsection{Main ideas of the proof}
The main ideas of the proof rely heavily on  symmetrization techniques and  normal form techniques.

Firstly,  to avoid losing derivatives, we introduce a good unknown
\beno
\vv u=\vv v+\epsilon\vv B^\epsilon(\z,\vv v),
\eeno
where $\vv B^\epsilon(\z,\vv v)$ is a nonlocal bilinear term which is defined at the beginning of Section 3.1. Here $\vv u$ is called the good unknown of Alinhac in \cite{ABZ,AM,Alin}. With $(\z,\vv u)$,
we symmetrize \eqref{Bsq2D} to the following dispersive equation
\beq\label{M1}
\p_tV-i\Lambda_\epsilon V=\mathcal{S}_V^\epsilon+\mathcal{Q}_V^\epsilon+O(\epsilon),
\eeq
where $\Lambda_\epsilon=|D|(1+\epsilon\D)$,  and $V=\z+i|D|^{-1}\div\,\vv u$  is the unknown which satisfies
\beno
\|V\|_{H^{N_0}}^2\sim\|\z\|_{H^{N_0}}^2+\|\vv v\|_{H^{N_0}}^2.
\eeno
 In \eqref{M1}, $\mathcal{S}_V^\epsilon$ is the symmetric quadratic term which is of order $O(\epsilon)$, $\mathcal{Q}_V^\epsilon$ is the quadratic term of order $O(\sqrt\epsilon)$, and term $O(\epsilon)$  contains all the remained  nonlinear terms of order $O(\epsilon)$.
The argument in this step is similar to the 1D Boussinesq in \cite{SX2}. A standard energy estimate leads to the well-posedness over time scalar $\f{1}{\sqrt\epsilon}$.

 The difference from the 1D Boussinesq is that there are two extra terms in $\mathcal{Q}_V^\epsilon$, i.e.,
\beno
Q_{+,+}^\epsilon(V^+,V^+),\quad Q_{-,+}^\epsilon(V^-,V^+),
\eeno
 where $V^+=V$ and $V^-=\overline{V}$. Fortunately, the delicate derivation shows that the symbols of these two terms are bounded by $\epsilon|\xi-\eta|$ which is the Fourier multiplier for the low frequency quantity. Then these two terms are also of order $O(\epsilon)$.
 The readers could refer to Section 3 for details.

To improve the bounds on the quadratic terms $Q_{+,-}^\epsilon(V^+,V^-)$ and $Q_{-,-}^\epsilon(V^-,V^-)$, we shall employ  normal forms techniques.  To sketch the main idea, we only consider the simple model
\beq\label{M2}
\p_tV-i\Lambda_\epsilon V=Q_{+,-}^\epsilon(V^+,V^-),
\eeq
where  $Q_{+,-}^\epsilon(V^+,V^-)$ is the quadratic term whose symbol $q_{+,-}^\epsilon(\xi,\eta)$ satisfies
\beq\label{M3}\begin{aligned}
&|q_{+,-}^\epsilon(\xi,\eta)|\lesssim\epsilon|\xi|\varphi_{\leq5}(\sqrt\epsilon|\eta|)\varphi_{\leq-6}\Bigl(\f{|\xi-\eta|}{|\eta|}\Bigr),\\
\text{and}\quad&\text{supp} (q_{+,-}^\epsilon)\subset\{(\xi,\eta)\in\R^2\times\R^2\,|\,\sqrt\epsilon|\eta|\leq64,\,\,\f{31}{32}|\eta|\leq|\xi|\leq\f{33}{32}|\eta|\}.
\end{aligned}\eeq
The definition of the symbol  $q_{+,-}^\epsilon(\xi,\eta)$ and the notation $\varphi_{\leq\cdot}$ can be found in   Subsections 2.1 and 2.2. \eqref{M3} yields that the rough bound of $Q_{+,-}^\epsilon(V^+,V^-)$ is of order $O(\sqrt\epsilon)$ which directly leads to the existence time of \eqref{M2} being of  order $O(\f{1}{\sqrt\epsilon})$.

Since non trivial zeroes of the phase occur in the set of moderate frequencies, instead of employing the normal form transformation directly for $Q_{+,-}^\epsilon(V^+,V^-)$,  we use  suitably modified normal form techniques when the integral regime is far away from the zero sets of the phase.

Assuming that
\beno
\|V(t)\|_{H^{N_0}}=O(1),\quad\text{for any } t\in[0,T_\epsilon],\quad T_\epsilon=O(\epsilon^{-\f23}),
\eeno
we only need to show
\beq\label{M4}
\|V(t)\|_{H^{N_0}}^2\lesssim1+\epsilon^{\f23} t.
\eeq
Then we could obtain the existence time of order $O(\epsilon^{-\f23})$ by a standard continuity argument. Indeed, an energy estimate for \eqref{M2} yields
\beno
\|V(t)\|_{H^{N_0}}^2\leq\|V(0)\|_{H^{N_0}}^2+\f{1}{(4\pi^2)^2}\Bigl|\underbrace{\int_0^t\int_{\R^2\times\R^2}\langle\xi\rangle^{2N_0}
q^\epsilon_{+,-}(\xi,\eta)\widehat{V^+}(\xi-\eta)\widehat{V^-}(\eta)
\overline{\widehat{V^+}(\xi)}d\eta d\xi dt}_{A}\bigr|.
\eeno

Defining the profiles $f$ and $g$ of $V$ and $\langle\na\rangle^{N_0}V$ as follows
\beno
f=e^{-it\Lambda_\epsilon}V\quad\text{and}\quad g=\langle\na\rangle^{N_0}f,
\eeno
we have
\beno
A=\int_0^t\int_{\R^2\times\R^2}\underbrace{e^{it\Phi^\epsilon_{+,-}(\xi,\eta)}\tilde{q}^\epsilon_{+,-}(\xi,\eta)\widehat{f^{+}}(\xi-\eta)\cdot
\widehat{g^-}(\eta)\cdot\widehat{g^-}(-\xi)}_{\mathcal{Q}^\epsilon(\xi,\eta)}d\eta d\xi d\tau,
\eeno
where
\beno
\Phi^\epsilon_{+,-}(\xi,\eta)=-\Lambda_\epsilon(\xi)+\Lambda_\epsilon(\xi-\eta)
-\Lambda_\epsilon(\eta),\quad\tilde{q}^\epsilon_{+,-}(\xi,\eta)=\langle\eta\rangle^{-N_0} \langle\xi\rangle^{N_0} q^\epsilon_{+,-}(\xi,\eta).
\eeno

By the definitions of profiles, we have
\beq\label{M5}
\|V\|_{H^{N_0}}\sim\|f\|_{H^{N_0}}\sim\|g\|_{L^2}\sim1.
\eeq
Thanks to Lemma \ref{phase lemma} and \eqref{M3}, we have
\beq\label{M6}\begin{aligned}
&\Phi^\epsilon_{+,-}(\xi,\eta)\sim|\eta|\phi_{+,-}^\epsilon(\xi,\eta),\quad\text{with }\phi_{+,-}^\epsilon(\xi,\eta)\,\,\text{defined in \eqref{phi + -}},\\
&|\tilde{q}^\epsilon_{+,-}(\xi,\eta)|\lesssim\epsilon|\xi|\varphi_{\leq5}(\sqrt\epsilon|\eta|)\varphi_{\leq-6}\Bigl(\f{|\xi-\eta|}{|\eta|}\Bigr)
\lesssim\sqrt\epsilon.
\end{aligned}\eeq
Moreover, since
\beno
\p_tf=e^{-it\Lambda_\epsilon}Q_{+,-}^\epsilon(V^+,V^-),
\eeno
using \eqref{M5} and \eqref{M6}, we have
\beq\label{M7}
\||D|^{-1}\p_tf\|_{H^{N_0}}=\||D|^{-1}\p_tg\|_{L^2}\lesssim\epsilon.
\eeq

According to the expressions of the phase $\Phi^\epsilon_{+,-}(\xi,\eta)$, generally, we shall divide the integral regime into two cases:  {\it phase far away from the spatial resonance set} and {\it phase near the spatial resonance set}. For the former case, we could use the normal formal techniques, that is integrating by parts  with respect to (w.r.t.) time $t$. While for the latter case, we shall use the smallness of the volume of the integral regime. However,  after similar arguments  as that for the 1D Boussinesq system in \cite{SX2}, we could not improve the existence time scale $\f{1}{\sqrt\epsilon}$. This is  because of the rough estimates over the latter regime, which is caused by the high dimension of the space. Therefore, to improve the existence time scale, we balance the size of symbol $\tilde{q}^\epsilon_{+,-}(\xi,\eta)$ and the volume of the integral regime when the phase near the spatial resonance set. To do so,  we compare the sizes of $|\xi-\eta|$ and $|\eta|$.

\medskip

Precisely, we divide the integral regime into the following three parts:

{\it (1). For low frequencies $\sqrt\epsilon|\eta|\leq\f12$,} there holds
\beno
|\phi^\epsilon_{+,-}(\xi,\eta)|\sim 1\quad\text{and}\quad
|\f{\tilde{q}^\epsilon_{+,-}(\xi,\eta)}{i\Phi^\epsilon_{+,-}(\xi,\eta)}|\lesssim\f{\epsilon}{|\phi^\epsilon_{+,-}(\xi,\eta)|}\lesssim \epsilon.
\eeno
Integrating by parts w.r.t.  $t$ and using \eqref{M5} and \eqref{M7}, we have
\beq\label{M8}
\bigl|\int_0^t\int_{R^2\times\R^2}\mathfrak{Q}^\epsilon(\xi,\eta)\varphi_{\leq-2}(\sqrt\epsilon|\eta|)d\eta d\xi dt\bigr|
\lesssim\epsilon+\epsilon^{\f32} t.
\eeq

{\it (2) For moderate frequencies with phase far away from the spatial resonance set,} i.e.,
\beno
\f14\leq\sqrt\epsilon|\eta|\leq64,\quad |\phi^\epsilon_{+,-}(\xi,\eta)|\geq 2^{-D-1},
\eeno
there holds
\beno
|\f{\tilde{q}^\epsilon_{+,-}(\xi,\eta)}{i\Phi^\epsilon_{+,-}(\xi,\eta)}|\lesssim\f{\epsilon}{|\phi^\epsilon_{+,-}(\xi,\eta)|}\lesssim 2^D\epsilon.
\eeno
Here $D\in\N$ is a large number which will be determined later on. Integrating by parts w.r.t.  $t$ and using \eqref{M5} and \eqref{M7}, we obtain
\beq\label{M9}
\bigl|\int_0^t\int_{\R^2\times\R^2}\mathfrak{Q}^\epsilon(\xi,\eta)\varphi_{[-1,5]}(\sqrt\epsilon|\eta|)
\varphi_{\geq-D}(\phi^\epsilon_{+,-}(\xi,\eta))d\eta d\xi dt\bigr|
\lesssim2^D\epsilon+2^D\epsilon^{\f32}t.
\eeq

{\it (3) For moderate frequencies with phase near the spatial resonance set,} i.e.,
\beno
\f14\leq\sqrt\epsilon|\eta|\leq64,\quad |\phi^\epsilon_{+,-}(\xi,\eta)|\leq 2^{-D},
\eeno
we shall split the integral regime into the following two parts
\beno
\f{|\xi-\eta|}{|\eta|}\leq 2^{-K+1}\quad\text{and}\quad\f{|\xi-\eta|}{|\eta|}\in[2^{-K},2^{-5}],
\eeno
where $K\in\N$ is a large number will be determined later on.

{\it (i) For case $\f{|\xi-\eta|}{|\eta|}\leq 2^{-K+1}$,} \eqref{M6} gives rise to
\beno
|\tilde{q}^\epsilon_{+,-}(\xi,\eta)|\lesssim\sqrt\epsilon,
\eeno
and
 Sine theorem yields
\beno
\angle(\xi,\eta)\approx\sin\bigl(\angle(\xi,\eta)\bigl)=\sin\bigl(\angle(\xi-\eta,\xi)\bigr)\f{|\xi-\eta|}{|\eta|}\leq2^{-K+1},
\eeno
where $\angle(\xi,\eta)$ is the angle between vectors $\xi$ and $\eta$.
Since the bound  of $|\tilde{q}^\epsilon_{+,-}(\xi,\eta)|$ is not small enough, we shall use the smallness of the volume of integral regime whose size is determined by the size of $\angle(\xi,\eta)$.

By localizing the angular $\angle(\xi,\eta)$, we could obtain
\beq\label{M10}\begin{aligned}
&\bigl|\int_0^t\int_{\R^2\times\R^2}\mathfrak{Q}^\epsilon(\xi,\eta)\varphi_{[-1,5]}(\sqrt\epsilon|\eta|)\cdot
\varphi_{\leq-D-1}(\phi^\epsilon_{+,-}(\xi,\eta))\cdot\varphi_{\leq-K}\bigl(\f{|\xi-\eta|}{|\eta|}\bigr)d\eta d\xi dt\bigr|\lesssim\sqrt\epsilon2^{-\f{K}{2}}t.
\end{aligned}\eeq
In \eqref{M10}, we obtained a small factor $2^{-\f{K}{2}}$ which is the contribution of the size of $\angle(\xi,\eta)$.

{\it (ii) For case $\f{|\xi-\eta|}{|\eta|}\in[2^{-K},2^{-5}]$,} \eqref{M6} yields
\beq\label{M12}
|\tilde{q}^\epsilon_{+,-}(\xi,\eta)|\lesssim\epsilon 2^K|\xi-\eta|.
\eeq
Since $|\xi-\eta|$ is a good Fourier multiplier for the low frequency quantity,  the bound of $|\tilde{q}^\epsilon_{+,-}(\xi,\eta)|$ has  a small coefficient $\epsilon 2^K$. Then using the volume of the integral regime whose size is determined by the size of the phase, we have
\beq\label{M11}\begin{aligned}
&\bigl|\int_0^t\int_{\R^2\times\R^2}\mathfrak{Q}^\epsilon(\xi,\eta)\varphi_{[-1,5]}(\sqrt\epsilon|\eta|)\cdot
\varphi_{\leq-D-1}(\phi^\epsilon_{+,-}(\xi,\eta))\cdot\varphi_{\geq-K+1}\bigl(\f{|\xi-\eta|}{|\eta|}\bigr)d\eta d\xi dt\bigr|\\
&\lesssim\epsilon^{\f34} 2^{\f{3K}{2}}2^{-\f{D}{2}}t.
\end{aligned}\eeq

Combining \eqref{M8}, \eqref{M9}, \eqref{M10} and \eqref{M11}, taking optimal $K$ and $D$, we obtain
 \beno
 |A|\lesssim\epsilon^{\f23} t.
 \eeno
Then we arrive at the energy estimate \eqref{M4}. The details of the proof are given in Section 4.

\begin{remark}
If we did not split the regime (3) into two parts: $
\f{|\xi-\eta|}{|\eta|}\leq 2^{-K+1}$  and $\f{|\xi-\eta|}{|\eta|}\in[2^{-K},2^{-5}]$, the integral over the regime (3) is bounded by
 $2^{-\f{D}{2}}t$. By the optimal choice $2^D\sim\epsilon^{-1}$, one only have
 \beno
 |A|\lesssim\sqrt\epsilon t.
 \eeno
 That is to say, the normal formal technique does not improve the energy estimates. Thus, we have to decompose the integral regime in a more flexible way.
\end{remark}

\setcounter{equation}{0}
\section{Preliminary}
\subsection{Definitions and notations}

 The notation $f\sim g$ means that there exists a constant $C$ such that $\f{1}{C}f\leq g\leq Cf$.  Notations $f\lesssim g$ and $g\gtrsim f$ mean that there exists a constant $C$ such that $f\leq Cg$. We shall use $C$ to denote a universal constant which may changes from line to line. For any $s\in\R$, $H^s(\R^2)$ denotes the classical  $L^2$ based Sobolev spaces with the norm $\|\cdot\|_{H^s}$.
The notation $\|\cdot\|_{L^p}$ stands for  the $L^p(\R^2)$ norm for $1\leq p \leq \infty$.

For vectors $\xi,\eta\in\R^2$, the notation $\angle(\xi,\eta)$ represents the angle between $\xi$ and $\eta$.

The $L^2(\R^2)$ scalar product is denoted by
$(u\,|\,v)_2\eqdefa\int_{\R^2}u\bar{v}dx$.

If  $A, B$ are two operators, $[A,B]=AB-BA$ denotes their commutator.

The Fourier transform of a tempered distribution $u\in\mathcal{S}'(\R^2)$ is denoted by $\widehat{u}$, which is defined as follows
\beno
\widehat{u}(\xi)\eqdefa\mathcal{F}(u)(\xi)=\int_{\R^2}e^{-ix\cdot\xi}u(x)dx.
\eeno
We use $\mathcal{F}^{-1}(f)$ to denote the inverse Fourier  transform of $f(\xi)$.

 If $f$ and $u$ are two functions defined on $\R^2$, the  Fourier multiplier  $f(D)u$  is defined in term of Fourier transform, i.e.,
\beno
\widehat{f(D)u}(\xi)=f(\xi)\widehat{u}(\xi).
\eeno

We shall use notations
\beno
\langle\xi\rangle=\bigl(1+|\xi|^2\bigr)^{\f12},\quad\langle\na\rangle=\bigl(1+|\na|^2\bigr)^{\f12}.
\eeno

For two well-defined functions $f(x)$, $g(x)$ and their bilinear form $Q(f,g)$, we use the convection that the symbol $q(\xi,\eta)$ of $Q(f,g)$  is defined in the following sense
\beno
\mathcal{F}\bigl(Q(f,g)\bigr)(\xi)=\f{1}{4\pi^2}\int_{\R^2}q(\xi,\eta)\hat{f}(\xi-\eta)\hat{g}(\eta)d\eta.
\eeno

\subsection{Para-differential decomposition theory}
Our proof of the main result relies on  suitable energy estimates for the solutions of \eqref{Bsq2D}. To do so, we introduce  para-differential formulations (see {\it e.g.,} \cite{BCD}) to symmetrize the system \eqref{Bsq2D}.

We fix an even smooth function $\varphi:\,\R\rightarrow [0,1]$ supported in $[-\f32,\f32]$ and equals to 1 in $[-\f54,\f54]$.  For any $k\in\Z$, we define
\beno
\varphi_k(x)\eqdefa\varphi(\f{x}{2^k})-\varphi(\f{x}{2^{k-1}}),\quad \varphi_{\leq k}(x)\eqdefa\varphi(\f{x}{2^k})=\sum_{l\leq k}\varphi_l(x).
\quad  \varphi_{\geq k}(x)\eqdefa1-\varphi_{\leq k-1}(x).
\eeno
While for any interval $I$ of $\R$, we define
\beno
\varphi_I(x)\eqdefa\sum_{k\in I}\varphi_k(x)=\sum_{k\in I\cap\Z}\varphi_k(x).
\eeno
Then for any $x\in\R$,
\beq\label{Bsq 11}
\sum_{k\in\Z}\varphi_k(x)=1\quad\text{and}\quad \text{supp}\,\varphi_k(\cdot)\subset\{x\in\R\,|\,|x|\in[\f{5}{8}2^k,\f{3}{2}2^k]\}.
\eeq

We use $P_k,\, P_{\leq k}$, $P_{\geq k}$ and $P_I$ to denote the Littlewood-Paley projection operators of the Fourier multiplier $\varphi_k,\,\varphi_{\leq k},\,\varphi_{\geq k}$ and $\varphi_I$, respectively.

We shall use the following para-differential decomposition: for any functions $f,g\in\mathcal{S}'(\R^2)$,
\beq\label{para-diff decomposition}
fg=T_fg+T_gf+R(f,g),
\eeq
with the para-differential operators being defined as follows
\beno\begin{aligned}
&T_fg=\sum_{j\in\Z}P_{\leq j-7}f\cdot P_jg,\quad R(f,g)=\sum_{j\in\Z}P_jf\cdot P_{[j-6,j+6]}g.
\end{aligned}\eeno

\subsection{Analysis of the phases}
In this subsection, we shall discuss the quadratic phase function
$\Phi_{\mu,\nu}^\epsilon(\xi,\eta)$ which is defined as follows:
\beq\label{quadratic phases}
\Phi_{\mu,\nu}^\epsilon(\xi,\eta)=-\Lambda_\epsilon(\xi)+\mu\Lambda_\epsilon(\xi-\eta)+\nu\Lambda_\epsilon(\eta),\quad\mu,\nu\in\{+,-\},
\eeq
where $\Lambda_\epsilon(\xi)$ is defined by
\beno
\Lambda_\epsilon(\xi)=(1-\epsilon|\xi|^2)|\xi|=|\xi|-\epsilon|\xi|^3.
\eeno

A direct calculation shows that
\beq\label{quadratic phases explicit}\begin{aligned}
\Phi_{\mu,\nu}^\epsilon(\xi,\eta)&=(|\xi|-\mu|\xi-\eta|-\nu|\eta|)\bigl[\epsilon\bigl(|\xi|^2+|\xi-\eta|^2+|\eta|^2
-\mu\nu|\xi-\eta||\eta|\\
&\qquad
+\mu|\xi||\xi-\eta|+\nu|\xi||\eta|\bigr)-1\bigr]+3\mu\nu\epsilon|\xi||\xi-\eta||\eta|.
\end{aligned}\eeq
Then we have the following lemma.

\begin{lemma}\label{phase lemma}
Assuming that $(\xi,\eta)\in\R^2\times\R^2$ satisfy $\xi\neq\eta,\,\xi\neq0,\,\eta\neq0$ and $\angle(\xi,\eta)\neq\pi$, we have
\beno\label{phase + -}
\Phi_{+,-}^\epsilon(\xi,\eta)=\f{|\xi||\eta|}{|\xi|+|\xi-\eta|+|\eta|}\phi_{+,-}^\epsilon(\xi,\eta),
\eeno
with
\beq\label{phi + -}\begin{aligned}
\phi_{+,-}^\epsilon(\xi,\eta)&=4\cos^2\bigl(\f12\angle(\xi,\eta)\bigr)\bigl[\epsilon\bigl(|\xi|^2+|\eta|^2-|\xi||\eta|\bigr)-1\bigr]\\
&\qquad+\epsilon\bigl(4\cos^2\bigl(\f12\angle(\xi,\eta)\bigr)-3\bigr)|\xi-\eta|(|\xi|+|\xi-\eta|+|\eta|),
\end{aligned}\eeq
and
\beno\label{phase - -}
\Phi_{-,-}^\epsilon(\xi,\eta)=\bigl(|\xi|+|\xi-\eta|+|\eta|\bigr)\phi_{-,-}^\epsilon(\xi,\eta),
\eeno
with
\beq\label{phi - -}\begin{aligned}
\phi_{-,-}^\epsilon(\xi,\eta)&=\epsilon\bigl(|\xi|^2+|\eta|^2-|\xi||\eta|\bigr)-1
+\epsilon\bigl(\f{3}{4\cos^2\bigl(\f12\angle(\xi,\eta)\bigr)}-1\bigr)|\xi-\eta|(|\xi|-|\xi-\eta|+|\eta|).
\end{aligned}\eeq
\end{lemma}
\begin{proof}
Since
\beno\begin{aligned}
&|\xi|-|\xi-\eta|+|\eta|=\f{(|\xi|+|\eta|)^2-|\xi-\eta|^2}{|\xi|+|\xi-\eta|+|\eta|}
 =\f{2|\xi||\eta|\bigl(1+\cos\bigl(\angle(\xi,\eta)\bigr)\bigr)}{|\xi|+|\xi-\eta|+|\eta|}
 =\f{4|\xi||\eta|\cos^2\bigl(\f12\angle(\xi,\eta)\bigr)}{|\xi|+|\xi-\eta|+|\eta|},\\
&\f{|\xi||\eta||\xi-\eta|}{|\xi|+|\xi-\eta|+|\eta|}=\f{|\xi||\eta||\xi-\eta|(|\xi|-|\xi-\eta|+|\eta|)}{(|\xi|+|\eta|)^2-|\xi-\eta|^2}
=\f{|\xi-\eta|(|\xi|-|\xi-\eta|+|\eta|)}{4\cos^2\bigl(\f12\angle(\xi,\eta)\bigr)},
\end{aligned}\eeno
\eqref{phase + -} and \eqref{phase - -} follow from \eqref{quadratic phases explicit}.
\end{proof}

\subsection{Bilinear estimates with the angle localized}
In order to improve the energy estimates near the spatial resonance set, we shall localize the angle $\angle(\xi,\eta)$ between $\xi$ and $\eta$ whose small size makes a crucial contribution in the energy estimates. To catch the contribution caused by the localized angle, we need the following bilinear estimate.

\begin{lemma}\label{biliear lemma}
Let $l,\,k,\,k_1,\,k_2\in\Z$,$|k-k_2|\leq2$, $l\leq-2$ and  $m(\xi,\eta)$ satisfy
\beno
\|m\|_{L^\infty_{k,k_1,k_2}}\eqdefa\|m(\xi,\eta)\|_{L^\infty_{\xi,\eta}(|\xi|\sim 2^k,\,|\xi-\eta|\sim 2^{k_1},\,|\eta|\sim2^{k_2})}<+\infty.
\eeno
For any $\mathfrak{f},\,\mathfrak{g}\in L^2(\R^2)$, defining a  bilinear form as follows
\beno
T_k(\mathfrak{f},\mathfrak{g})(\xi)=\int_{\R^2}m(\xi,\eta)\hat{\mathfrak{f}}(\xi-\eta)\hat{\mathfrak{g}}(\eta)\varphi_{k}(\xi)\varphi_{k_1}(\xi-\eta)
\varphi_{k_2}(\eta)\varphi_{l}(\angle(\xi,\eta))d\eta,
\eeno
we have
\beq\label{biliear estimate}
\|T_k(\mathfrak{f},\mathfrak{g})(\xi)\|_{L^2}\lesssim2^{\f{l}{2}}2^{k_1}\|m\|_{L^\infty_{k,k_1,k_2}}\cdot \|\mathfrak{f}_{k_1}\|_{L^2}\cdot\|\mathfrak{g}_{k_2}\|_{L^2},
\eeq
where $\mathfrak{f}_{k_1}=\mathcal{F}^{-1}\bigl[\hat{\mathfrak{f}}(\xi)\varphi_{k_1}(\xi)\bigr]$,
$\mathfrak{g}_{k_2}=\mathcal{F}^{-1}\bigl[\hat{\mathfrak{g}}(\xi)\varphi_{k_2}(\xi)\bigr]$.
\end{lemma}
\begin{remark}
The main contribution of the bilinear estimate \eqref{biliear estimate} is to gain  the factor $2^{\f{l}{2}}$.
\end{remark}
\begin{proof}
Firstly,  for a given small number $2^n\in(0,8)$ $(n\in\Z)$, we decompose the unit circle $\mathbb{S}^1$ into the union of $N_n$ angular sectors, and each sector has angular size $2^n$, where $N_n=O(2^{-n})$. The number of the overlaps is bounded by a universal number $n_0\in\N$ that is independent of $n$. We use notation $\omega\in\mathbb{S}^1$ as the angular vector and $\D\omega$ as the size of each angular sector. Then there exists a partition of unity $\{b_n^{j}(\omega)\}=\{b_n^{j}(\omega)\}_{j=1,\cdots, N_n}$ corresponding to the decomposition of $\mathbb{S}^1$.

We define
the angular vectors of $\xi,\,\eta,\xi-\eta\in\R^2\setminus\{0\}$ as follows
\beno
\omega_\xi\eqdefa\f{\xi}{|\xi|}\in\mathbb{S}^1,\quad\omega_\eta\eqdefa\f{\eta}{|\eta|}\in\mathbb{S}^1,
\quad\omega_{\xi-\eta}\eqdefa\f{\xi-\eta}{|\xi-\eta|}\in\mathbb{S}^1.
\eeno
By virtue of the Sine theorem, we have the angular partition of unity $\{b_{k_2-k_1+l}^{j_\xi}(\omega_\xi)\}$,
$\{b_l^{j_{\xi-\eta}}(\omega_{\xi-\eta})\}$ and $\{b_{k_2-k_1+l}^{j_\eta}(\omega_\eta)\}$  for $\xi,\xi-\eta,\eta$, where
\beq\label{angular partitions}\begin{aligned}
&|\omega_\xi-\omega_\eta|\sim\angle(\xi,\eta)\sim 2^l,\\
& |\omega_{\xi-\eta}-\omega_\eta|\sim\angle(\xi-\eta,\eta)\sim 2^{k_2-k_1+l}\,\,
\text{or}\,\,|\omega_{\xi-\eta}+\omega_\eta|\sim\angle(\xi-\eta,-\eta)\sim 2^{k_2-k_1+l},\\
&\D\omega_{\xi-\eta}\sim2^l,\quad\D\omega_\xi\sim2^{k_2-k_1+l},\quad \D\omega_\eta\sim2^{k_2-k_1+l}.
\end{aligned}\eeq
Since $|\xi-\eta|\geq|\eta|\sin\bigl(\angle(\xi,\eta)\bigr)\approx|\eta|\angle(\xi,\eta)$, there holds $k_2-k_1+l\leq4$ so that the decompositions in \eqref{angular partitions} are reasonable.
Then we have
\beq\label{B13}\begin{aligned}
T_k(\mathfrak{f},\mathfrak{g})(\xi)&=\sum_{j_\xi,j_{\xi-\eta},j_{\eta}}\int_{\R^2}m(\xi,\eta)\widehat{\mathfrak{f}_{k_1}}(\xi-\eta)
\widehat{\mathfrak{g}_{k_2}}(\eta)\varphi_{k}(\xi)\\
&\qquad\cdot b_{k_2-k_1+l}^{j_\xi}(\omega_\xi)b_l^{j_{\xi-\eta}}(\omega_{\xi-\eta})
b_{k_2-k_1+l}^{j_\eta}(\omega_\eta)\varphi_{l}(\angle(\xi,\eta))d\eta.
\end{aligned}\eeq
We remark that $j_\xi,j_{\xi-\eta},j_{\eta}$ in \eqref{B13} are restricted by \eqref{angular partitions}.

Thanks to the $L^2$-orthogonality of $\{b_{k_2-k_1+l}^{j_\xi}(\omega_\xi)\}$, we have
\beq\label{B14}\begin{aligned}
&\|T_k(\mathfrak{f},\mathfrak{g})\|_{L^2_\xi}^2\lesssim\sum_{j_\xi,j_{\xi-\eta},j_{\eta}}
\|\int_{\R^2}m(\xi,\eta)\widehat{\mathfrak{f}_{k_1}}(\xi-\eta)
\widehat{\mathfrak{g}_{k_2}}(\eta)\varphi_{k}(\xi)\\
&\qquad\cdot b_{k_2-k_1+l}^{j_\xi}(\omega_\xi)b_l^{j_{\xi-\eta}}(\omega_{\xi-\eta})
b_{k_2-k_1+l}^{j_\eta}(\omega_\eta)\varphi_{l}(\angle(\xi,\eta))d\eta\|_{L^2_\xi}^2\\
&\lesssim\|m\|_{L^\infty_{k,k_1,k_2}}^2
\sum_{j_\xi,j_{\xi-\eta},j_{\eta}}\|\widehat{\mathfrak{f}_{k_1}}(\xi-\eta)b_l^{j_{\xi-\eta}}(\omega_{\xi-\eta})\|_{L^1_{\xi-\eta}}^2
\|\widehat{\mathfrak{g}_{k_2}}(\eta)b_l^{j_\eta}(\omega_{\eta})\|_{L^2_\eta}^2\\
&\lesssim2^{2k_1+l}\|m\|_{L^\infty_{k,k_1,k_2}}^2
\sum_{j_{\xi-\eta},j_{\eta}}\sum_{j_\xi}\|\widehat{\mathfrak{f}_{k_1}}(\xi-\eta)b_l^{j_{\xi-\eta}}(\omega_{\xi-\eta})\|_{L^2_{\xi-\eta}}^2
\|\widehat{\mathfrak{g}_{k_2}}(\eta)b_l^{j_\eta}(\omega_{\eta})\|_{L^2_\eta}^2.
\end{aligned}\eeq
In the last inequality of \eqref{B14}, we used the volume of the integral regime.

Due to \eqref{angular partitions}, for fixed $j_{\xi-\eta}$, there are finite $j_{\eta}$ such that \eqref{angular partitions} holds. While for fixed
$j_\eta$, there are finite $j_\xi$ such that \eqref{angular partitions} holds. Using the $L^2$-orthogonality of decompositions for the angle of $\xi-\eta$ and $\eta$, we deduce from \eqref{B14} that
\beno
\|T_k(\mathfrak{f},\mathfrak{g})\|_{L^2_\xi}
\lesssim2^{k_1+\f{l}{2}}\|m\|_{L^\infty_{k,k_1,k_2}}\|\mathfrak{f}_{k_1}\|_{L^2}
\|\mathfrak{g}_{k_2}\|_{L^2}.
\eeno
This is exactly \eqref{biliear estimate}. The lemma is proved.
\end{proof}

\medskip

We end up this section with the following commutator estimate.

\begin{lemma}\label{commutator lem 2} Let $s>-1$, $\na a\in L^\infty(\R^2)$ and $b\in H^{s-1}(\R^2)$. There holds
\beq\label{commutator}
\|[|D|^{-1}\div,T_a]b\|_{H^s}\lesssim\|\na a\|_{L^\infty}\|b\|_{H^{s-1}}.
\eeq
\end{lemma}

The proof of Lemma \ref{commutator lem 2} follows from the definition of $T_ab$ and Theorem 3 of \cite{La2}.

\setcounter{equation}{0}
\section{Symmetrization of system \eqref{Bsq2D}}
In this section, we shall symmetrize system \eqref{Bsq2D} by using para-differential decomposition and introducing good unknowns. Then we state a main proposition on the symmetric system.

\subsection{Symmetrization of system \eqref{Bsq2D}}
Firstly, we introduce a good unknown $\vv u$ with
\beq\label{good unknowns for 2D}
\vv u=\vv v+\epsilon\vv B^\epsilon(\z,\vv v),
\eeq
where $\vv B^\epsilon(\cdot,\cdot)$ is a bilinear operator defined as
\beno
\vv B^\epsilon(f,\vv g)=\f12T_f\bigl((1+\epsilon\D)^{-1}\varphi_{\geq6}(\sqrt\epsilon|D|)\vv g\bigr).
\eeno
Here $\vv u$ is called the good unknown in the sense of Alinhac in \cite{ABZ,AM,Alin}.
Without confusion, we sometimes use $\vv B^\epsilon$ to denote the bilinear term $\vv B^\epsilon(\z,\vv v)$.

With $(\z,\vv u)$, we rewrite the first equation of \eqref{Bsq2D} as
\beno\begin{aligned}
\p_t\z+(1+\epsilon\D)\na\cdot\vv u&=-\epsilon \na\cdot(\z\vv v)+\epsilon(1+\epsilon\D)\na\cdot\vv B^\epsilon(\z,\vv v)\\
&=-\epsilon \na\cdot(T_{\vv v}\z)-\epsilon\na\cdot(T_\z\vv v)-\epsilon\na\cdot\bigl(R(\z,\vv v)\bigr)+\epsilon(1+\epsilon\D)\na\cdot\vv B^\epsilon(\z,\vv v).
\end{aligned}\eeno
Since
\beno\begin{aligned}
&\epsilon(1+\epsilon\D)\na\cdot\vv B^\epsilon=\f{\epsilon}{2}\na\cdot\bigl(T_\z \varphi_{\geq6}(\sqrt{\epsilon}|D|)\vv v\bigr)+\f{\epsilon^2}{2}\na\cdot\Bigl([\D,T_\z](1+\epsilon\D)^{-1}\varphi_{\geq6}(\sqrt{\epsilon}|D|)\vv v\Bigr),
\end{aligned}\eeno
we have
\beq\label{equation for zeta}\begin{aligned}
\p_t\z+(1+\epsilon\D)\na\cdot\vv u+\epsilon \na\cdot(T_{\vv v}\z)+\f{\epsilon}{2}\na\cdot(T_\z\vv u)=-\f{\epsilon}{2}\na\cdot\bigl(T_\z \varphi_{\leq5}(\sqrt\epsilon|D|)\vv v\bigr)+N_\z^\epsilon,
\end{aligned}\eeq
where
\beno\begin{aligned}
N_\z^\epsilon&=-\epsilon\na\cdot\bigl(R(\z,\vv v)\bigr)
+\f{\epsilon^2}{2}\na\cdot\Bigl([\D,T_\z](1+\epsilon\D)^{-1}\varphi_{\geq6}(\sqrt{\epsilon}|D|)\vv v\Bigr)+\f{\epsilon^2}{2}\na\cdot\bigl(T_\z \vv B^\epsilon\bigr).
\end{aligned}\eeno

Due to \eqref{curl free condition}, we have
\beq\label{B1}
\f{\epsilon}{2}\na\bigl(|\vv v|^2\bigr)=\epsilon\vv v\cdot\na\vv v=\epsilon T_{\vv v}\cdot\na\vv v+\epsilon T_{\na\vv v}\cdot\vv v+\epsilon R(\vv v\cdot,\na\vv v).
\eeq
Using \eqref{Bsq2D}, \eqref{good unknowns for 2D} and \eqref{B1}, we have
\beno\begin{aligned}
&\p_t\vv u+(1+\epsilon\D)\na\z=\p_t\vv v+(1+\epsilon\D)\na\z+\epsilon\vv B^\epsilon(\p_t\z,\vv v)+\epsilon\vv B^\epsilon(\z,\p_t\vv v)\\
&=-\epsilon T_{\vv v}\cdot\na\vv v-\epsilon T_{\na\vv v}\cdot\vv v-\epsilon R(\vv v\cdot,\na\vv v)-\epsilon \vv B^\epsilon\bigl((1+\epsilon\D)\na\cdot\vv v,\vv v\bigr)\\
&\quad-\epsilon^2 \vv B^\epsilon\bigl(\na\cdot(\z\vv v),\vv v\bigr)-\epsilon\vv B^\epsilon\bigl(\z,(1+\epsilon\D)\na\z\bigr)-\f{\epsilon^2}{2}\vv B^\epsilon(\z,\na(|\vv v|^2)).
\end{aligned}\eeno

Due to the definition of $\vv B^\epsilon(\cdot,\cdot)$, we have
\beno
\vv B^\epsilon\bigl(\z,(1+\epsilon\D)\na\z\bigr)=\f12T_\z\bigl(\varphi_{\geq6}(\sqrt\epsilon|D|)\na\z\bigr),
\eeno
which implies
\beq\label{equation for u}
\p_t\vv u+(1+\epsilon\D)\na\z+\epsilon T_{\vv v}\cdot\na\vv u+\f{\epsilon}{2}\na(T_\z\z)=\f{\epsilon}{2}\na\bigl(T_\z\varphi_{\leq5}(\sqrt\epsilon|D|)\z\bigr)
+N_{\vv u}^\epsilon,
\eeq
where
\beno\begin{aligned}
N_{\vv u}^\epsilon&=\f{\epsilon}{2}T_{\na\z}\varphi_{\geq6}(\sqrt\epsilon|D|)\z-\epsilon T_{\na\vv v}\cdot\vv v-\epsilon R(\vv v\cdot,\na\vv v)
-\epsilon\vv B^\epsilon\bigl((1+\epsilon\D)\na\cdot\vv v,\vv v\bigr)\\
&\quad+\epsilon^2T_{\vv v}\cdot\na\vv B^\epsilon-\epsilon^2\vv B^\epsilon\bigl(\na\cdot(\z\vv v),\vv v\bigr)-\f{\epsilon^2}{2}\vv B^\epsilon\bigl(\z,\na(|\vv v|^2)\bigr).
\end{aligned}\eeno

Now, defining
\beq\label{good variable for 2D}
V=\z+i|D|^{-1}\div\vv u.
\eeq
we deduce from \eqref{equation for zeta} and \eqref{equation for u} that
\beq\label{New Bsq for 2D}\begin{aligned}
&\p_tV-i\Lambda_\epsilon V+\epsilon\na\cdot(T_vV)-\f{i\epsilon}{2}|D|(T_\z V)\\
&=-\f{\epsilon}{2}\na\cdot\bigl(T_\z \varphi_{\leq5}(\sqrt\epsilon|D|)\vv v\bigr)-\f{i\epsilon}{2}|D|\bigl(T_\z\varphi_{\leq5}(\sqrt\epsilon|D|)\z\bigr)+N_V^\epsilon,
\end{aligned}\eeq
where $\Lambda_\epsilon=|D|(1+\epsilon\D)$ and
\beno
N_V^\epsilon=-i\epsilon[|D|^{-1}\div,\,T_{\vv v}]\cdot\na\vv u+i\epsilon T_{\div\vv v}(|D|^{-1}\div\vv u)
-\f{\epsilon}{2}|D|\bigl([|D|^{-1}\div,\,T_\z]\vv u\bigr)+N_\z^\epsilon+i|D|^{-1}\div N_{\vv u}^\epsilon.
\eeno

\begin{remark}
The left hand side (l.h.s) of \eqref{New Bsq for 2D} is the quasi-linear part of system \eqref{Bsq2D}. The first two terms in the right hand side (r.h.s) of  \eqref{New Bsq for 2D} are the worst quadratic terms which are of order $O(\sqrt\epsilon)$. The nonlinear terms in $N_V^\epsilon$ are of order $O(\epsilon)$.
\end{remark}

We also remark that we could deduce from the evolution equations of $\z$ and $\vv u$ that
\beq\label{linearized system}
\p_tV-i\Lambda_\epsilon V=-\epsilon \na\cdot(\z\vv v)+\epsilon(1+\epsilon\D)\na\cdot\vv B^\epsilon
+i\epsilon|D|^{-1}\div\bigl[-\f{1}{2}\na(|\vv v|^2)+\vv B^\epsilon(\p_t\z,\vv v)+\vv B^\epsilon(\z,\p_t\vv v)\bigr].
\eeq

\medskip

 Denoting by
\beno
V^+=V,\quad V^-=\overline{V},
\eeno
we shall rewrite the quadratic terms of \eqref{New Bsq for 2D} in terms of $V^+$ and $V^-$. Whereas we keep the remaining  nonlinear terms in $N_V^\epsilon$ in terms of $\z$ and $v$.

By the definition of $V$, we have
\beno
\z=\f12(V^++V^-)=\f12\sum_{\mu\in\{+,-\}}V^\mu,\quad\div\vv u=-\f{i}{2}|D|(V^+-V^-)=-\f{i}{2}|D|\sum_{\nu\in\{+,-\}}\nu V^\nu.
\eeno
Since
\beno
\div\vv u=\div\vv v+\epsilon\div\vv B^\epsilon,\quad\curl\vv v=0,
\eeno
we have
\beno
\vv v=\D^{-1}\na\bigl(\div\vv u-\epsilon\div\vv B^\epsilon\bigr)=\f{i}{2}\sum_{\nu\in\{+,-\}}\nu |D|^{-1}\na V^\nu-\epsilon\D^{-1}\na\div\vv B^\epsilon.
\eeno
Then
we have
\beq\label{zeta v and u}\begin{aligned}
&\z=\f12\sum_{\mu\in\{+,-\}}V^\mu,\quad \wt{\vv v}=\f{i}{2}\sum_{\nu\in\{+,-\}}\nu |D|^{-1}\na V^\nu,\\
&\vv v=\wt{\vv v}-\epsilon\D^{-1}\na\div\vv B^\epsilon,\quad\vv u=\wt{\vv v}-\epsilon\D^{-1}\na\div\vv B^\epsilon+\epsilon\vv B^\epsilon.
\end{aligned}\eeq

Before ending this subsection, we provide a lemma involving the bilinear operator $\vv B^\epsilon(\cdot,\cdot)$.

\begin{lemma}\label{lem for Beps}
Assume that the real-valued function $f\in L^\infty(\R^2)$ and vector function $\vv g\in H^s(\R^2)$ for $s\geq -2$. There hold
\beq\label{estimate for Beps}
\mathcal{F}\bigl(\vv B^\epsilon(f,\vv g)\bigr)(\xi)=\overline{\mathcal{F}\bigl(\vv B^\epsilon(f,\vv g)\bigr)(-\xi)}
\eeq
and for $k=0,1,2$,
\beq\label{estimate for Beps a}\begin{aligned}
&\|\vv B^\epsilon(f,\vv g)\|_{H^{s+k}}\leq C_B\epsilon^{-\f{k}{2}}\|f\|_{L^\infty}\|\vv g\|_{H^s},
\end{aligned}\eeq
where $C_B>0$ is a universal constant.
\end{lemma}
The proof of the lemma is similar to that of Lemma 3.1 in \cite{SX2}. We omit the proof here.

\subsection{Main proposition for the symmetric system \eqref{New Bsq for 2D}}
Now, we state the main proposition of this subsection.
\begin{proposition}\label{New Bsq 2D prop}
Assume that $(\z,\vv v)\in H^{N_0}(\R^2)$ with $N_0\geq 5$ solves \eqref{Bsq2D}. Then $V$ defined in \eqref{good variable for 2D} satisfies the following system
\beq\label{symmetric formula}
\p_tV-i\Lambda_\epsilon V=\mathcal{S}_V^\epsilon+\mathcal{Q}_V^\epsilon+\mathcal{C}_V^\epsilon+\mathcal{N}_V^\epsilon,
\eeq
where
\begin{itemize}
\item The quadratic term $\mathcal{S}_V^\epsilon$ is of the form
\beno
\mathcal{S}_V^\epsilon=\sum_{\mu\in\{+,-\}}S_{\mu,+}^\epsilon(V^\mu,V^+),
\eeno
with the symbol $s_{\mu,+}^\epsilon(\xi,\eta)$ of $S_{\mu,+}^\epsilon$   satisfying
\beq\label{symbol s eps}
\overline{s_{\mu,+}^\epsilon(\xi,\eta)}=-s_{\mu,+}^\epsilon(\xi,\eta),
\eeq
\beq\label{bound of s eps}\begin{aligned}
&|\langle\xi\rangle^{-N_0}\langle\eta\rangle^{-N_0}\bigl(\langle\xi\rangle^{2N_0}s_{\mu,+}^\epsilon(\xi,\eta)
-\langle\eta\rangle^{2N_0}s_{-\mu,+}^\epsilon(\eta,\xi)\bigr)|\lesssim\epsilon|\xi-\eta|\cdot\varphi_{\leq -6}\Bigl(\f{|\xi-\eta|}{\max\{|\xi|,|\eta|\}}\Bigr).
\end{aligned}\eeq
\item The quadratic term $\mathcal{Q}_V^\epsilon$ is of the form
\beno
\mathcal{Q}_V^\epsilon=\sum_{\mu,\nu\in\{+,-\}}Q_{\mu,\nu}^\epsilon(V^\mu,V^\nu),
\eeno
with the symbol $q_{\mu,\nu}^\epsilon(\xi,\eta)$ of $Q_{\mu,\nu}^\epsilon$  satisfying
\beq\label{bound of q eps}\begin{aligned}
&|q_{\mu,-}^\epsilon(\xi,\eta)|\lesssim\epsilon|\xi|\cdot\varphi_{\leq 5}\bigl(\sqrt{\epsilon}|\eta|\bigr)\cdot\varphi_{\leq -6}\Bigl(\f{|\xi-\eta|}{|\eta|}\Bigr),\\
&|q_{\mu,+}^\epsilon(\xi,\eta)|\lesssim\epsilon|\xi-\eta|\cdot\varphi_{\leq 5}\bigl(\sqrt{\epsilon}|\eta|\bigr)\cdot\varphi_{\leq -6}\Bigl(\f{|\xi-\eta|}{|\eta|}\Bigr)
\end{aligned}\eeq
\item The cubic term $\mathcal{C}_V^\epsilon=\epsilon^2\na\cdot(T_{\D^{-1}\na\div\vv B^\epsilon}V)$
satisfies
\beq\label{estimate for cubic term eps}
\bigl|\text{Re}\bigl\{\bigl(\langle\na\rangle^{N_0}\mathcal{C}_V^\epsilon\,|\,\langle\na\rangle^{N_0}V\bigr)_2\bigr\}\bigr|
\lesssim\epsilon^2\|\z\|_{H^{N_0}}\|\vv v\|_{H^{N_0}}\|V\|_{H^{N_0}}^2.
\eeq
\item The nonlinear term $\mathcal{N}_V^\epsilon$  satisfies
\beq\label{estimate for remain terms}
\|\mathcal{N}_V^\epsilon\|_{H^{N_0}}\lesssim\epsilon\bigl(1+\|\z\|_{H^{N_0}}\bigr)\bigl(\|\z\|_{H^{N_0}}^2+\|\vv v\|_{H^{N_0}}^2+\|\vv u\|_{H^{N_0}}^2\bigr).
\eeq
\end{itemize}
\end{proposition}

\begin{remark}
Proposition \ref{New Bsq 2D prop} suggests that there is no loss of derivative  for the nonlinear terms of \eqref{New Bsq for 2D}. Indeed, in the energy estimates, we shall use the symmetric structure of the quadratic terms $\mathcal{S}_V^\epsilon$ to avoid the loss of derivatives (see \eqref{bound of s eps}). We also use the symmetric structure of $\mathcal{C}_V^\epsilon$ to avoid losing derivative (see the proof of \eqref{estimate for cubic term eps}). The quadratic terms $\mathcal{S}_V^\epsilon$, $Q_{\mu,+}^\epsilon$ and the nonlinear term $\mathcal{N}_V^\epsilon$ are of order $O(\epsilon)$. Whereas the quadratic term $Q_{\mu,-}^\epsilon$ is of order $O(\sqrt\epsilon)$ which directly leads to the existence time of order $O(\f{1}{\sqrt\epsilon})$. To enlarge the existence time scalar $\f{1}{\sqrt\epsilon}$, we shall use  norm formal techniques to deal with the worst quadratic term $Q_{\mu,-}^\epsilon$.
\end{remark}

\begin{proof}[ Proof of Proposition \ref{New Bsq 2D prop}]
The nonlinear terms in the r.h.s. of \eqref{symmetric formula} stem from the nonlinear terms in \eqref{New Bsq for 2D}. When we rewrite \eqref{New Bsq for 2D} to \eqref{symmetric formula}, using \eqref{zeta v and u}, the nonlinear terms of \eqref{symmetric formula} firstly read as
\beno\begin{aligned}
&\mathcal{S}_V^\epsilon=-\epsilon\na\cdot(T_{\wt{\vv v}}V)+\f{i\epsilon}{2}|D|(T_\z V),\\
&\mathcal{Q}_V^\epsilon=-\f{\epsilon}{2}\na\cdot\bigl(T_\z \varphi_{\leq5}(\sqrt\epsilon|D|)\wt{\vv v}\bigr)-\f{i\epsilon}{2}|D|\bigl(T_\z \varphi_{\leq5}(\sqrt\epsilon|D|)\z\bigr),\\
&\mathcal{C}_V^\epsilon=\epsilon^2\na\cdot(T_{\D^{-1}\na\div\vv B^\epsilon}V),\\
&\mathcal{N}_V^\epsilon=\f{\epsilon^2}{2}\na\cdot\bigl(T_\z\varphi_{\leq5}(\sqrt\epsilon|D|)\D^{-1}\na\div\vv B^\epsilon\bigr)+N^\epsilon_V.
\end{aligned}\eeno

\medskip

(1) {\it For the quadratic term $\mathcal{S}_V^\epsilon$,}
by virtue of \eqref{zeta v and u}, we rewrite it in terms of $V^+$ and $V^-$ as
\beno
\mathcal{S}_V^\epsilon=\sum_{\mu\in\{+,-\}}S_{\mu,+}^\epsilon(V^\mu,V^+)
\eeno
with
\beno\begin{aligned}
&S_{\mu,+}^\epsilon(V^\mu,V^+)=-\mu\f{i\epsilon}{2}\na\cdot(T_{|D|^{-1}\na V^\mu}V^+)+\f{i\epsilon}{4}|D|(T_{V^\mu} V^+).
\end{aligned}\eeno
By the definition of the para-differential operator, we have
\beno
\mathcal{F}\Bigl(S_{\mu,+}^\epsilon(V^\mu,V^+)\Bigr)(\xi)=\f{1}{4\pi^2}\int_{\R^2}s_{\mu,+}^\epsilon(\xi,\eta)\widehat{V^\mu}(\xi-\eta)\widehat{V^+}(\eta)d\eta.
\eeno
with the symbol  $s_{\mu,+}^\epsilon(\xi,\eta)$ of $S_{\mu,+}^\epsilon(V^\mu,V^+)$ being as
\beq\label{symbol s 2D}
s_{\mu,+}^\epsilon(\xi,\eta)=i\epsilon\bigl(\mu\f12\xi\cdot(\xi-\eta)|\xi-\eta|^{-1}+\f14|\xi|\bigr)
\sum_{j\in\Z}\varphi_{\leq j-7}(|\xi-\eta|)\varphi_j(|\eta|).
\eeq
Then \eqref{symbol s 2D} yields $\overline{s_{\mu,+}^\epsilon(\xi,\eta)}=-s_{\mu,+}^\epsilon(\xi,\eta)$ which is exactly \eqref{symbol s eps}.
By a direct derivation, it is easy to check that
\beq\label{B2}
\sum_{j\in\Z}\varphi_{\leq j-7}(|\xi-\eta|)\varphi_j(|\eta|)\lesssim\varphi_{\leq -6}\Bigl(\f{|\xi-\eta|}{|\eta|}\Bigr).
\eeq
Since
\beno\begin{aligned}
&\quad\langle\xi\rangle^{2N_0}s_{\mu,+}^\epsilon(\xi,\eta)
-\langle\eta\rangle^{2N_0}s_{-\mu,+}^\epsilon(\eta,\xi)\\
&=i\epsilon\langle\xi\rangle^{2N_0}
\Bigl(\mu\f12\xi\cdot(\xi-\eta)|\xi-\eta|^{-1}+\f{1}{4}|\xi|\Bigr)\sum_{j\in\Z}\varphi_{\leq j-7}(|\xi-\eta|)\varphi_j(|\eta|)\\
&\qquad
-i\epsilon\langle\eta\rangle^{2N_0}\Bigl(\mu\f12\eta\cdot(\xi-\eta)|\xi-\eta|^{-1}+\f{1}{4}|\eta|\Bigr)\sum_{j\in\Z}\varphi_{\leq j-7}(|\xi-\eta|)\varphi_j(|\xi|),
\end{aligned}\eeno
and
\beno\begin{aligned}
&|\varphi_j(|\xi|)-\varphi_j(|\eta|)|\lesssim\f{1}{\min\{|\xi|,|\eta|\}}|\xi-\eta|,
\end{aligned}\eeno
using \eqref{B2}, we obtain
\beno
|\langle\xi\rangle^{-N_0}\langle\eta\rangle^{-N_0}\bigl(\langle\xi\rangle^{2N_0}s_{\mu,+}^\epsilon(\xi,\eta)
-\langle\eta\rangle^{2N_0}s_{-\mu,+}^\epsilon(\eta,\xi)\bigr)|\lesssim\epsilon|\xi-\eta|\cdot\varphi_{\leq -6}\Bigl(\f{|\xi-\eta|}{\max\{|\xi|,|\eta|\}}\Bigr).
\eeno
This is exactly \eqref{bound of s eps}.
\medskip

(2) {\it For the quadratic term $\mathcal{Q}_V^\epsilon$,}
by virtue of \eqref{zeta v and u}, we rewrite it in terms of $V^+$ and $V^-$ as
\beno
\mathcal{Q}_V^\epsilon=\sum_{\mu,\nu\in\{+,-\}}Q_{\mu,\nu}^\epsilon(V^\mu,V^\nu),
\eeno
where
\beno
Q_{\mu,\nu}^\epsilon(V^\mu,V^\nu)=-\nu\f{i\epsilon}{8}\na\cdot\bigl(T_{V^\mu}\varphi_{\leq5}(\sqrt\epsilon|D|)|D|^{-1}\na V^\nu\bigr)
-\f{i\epsilon}{8}|D|\bigl(T_{V^\mu}\varphi_{\leq5}(\sqrt\epsilon|D|)V^\nu\bigr).
\eeno
Applying Fourier transformation to $Q_{\mu,\nu}^\epsilon(V^\mu,V^\nu)$, we have
\beno\begin{aligned}
\mathcal{F}\Bigl(Q_{\mu,\nu}^\epsilon(V^\mu,V^\nu)\Bigr)(\xi)=\f{1}{4\pi^2}
\int_{\R^2}q_{\mu,\nu}^\epsilon(\xi,\eta)\widehat{V^\mu}(\xi-\eta)\widehat{V^\nu}(\eta)d\eta,
\end{aligned}\eeno
with the symbol $q_{\mu,\nu}^\epsilon(\xi,\eta)$ of $Q_{\mu,\nu}^\epsilon(V^\mu,V^\nu)$  being as follows
\beq\label{symbol q eps}
q_{\mu,\nu}^\epsilon(\xi,\eta)=\f{i\epsilon}{8}\xi\cdot\bigl(\nu\f{\eta}{|\eta|}-\f{\xi}{|\xi|}\bigr)
\varphi_{\leq5}(\sqrt\epsilon|\eta|)
\sum_{j\in\Z}\varphi_{\leq j-7}(|\xi-\eta|)\varphi_j(|\eta|).
\eeq

Thanks to \eqref{B2} and \eqref{symbol q eps}, we obtain
\beno
|q_{\mu,-}^\epsilon(\xi,\eta)|\lesssim\epsilon|\xi|\varphi_{\leq5}(\sqrt\epsilon|\eta|)\varphi_{\leq-6}\Bigl(\f{|\xi-\eta|}{|\eta|}\Bigr),
\eeno
which is the first part of \eqref{bound of q eps}. And we also have
\beq\label{B3}
|q_{\mu,+}^\epsilon(\xi,\eta)|\lesssim\epsilon|\xi|\angle(\xi,\eta)
\varphi_{\leq5}(\sqrt\epsilon|\eta|)\varphi_{\leq-6}\Bigl(\f{|\xi-\eta|}{|\eta|}\Bigr),
\eeq
where we used the fact that $\Bigl|\f{\eta}{|\eta|}-\f{\xi}{|\xi|}\Bigr|\sim\angle(\xi,\eta)$.
Moreover, we have
\beq\label{supp of q}\begin{aligned}
\mathrm{supp}\,(q_{\mu,\nu}^\epsilon)&\subset\{(\xi,\eta)\in\R^2\times\R^2\,|\,\sqrt\epsilon|\eta|\leq 64,\,|\xi-\eta|\leq\f{1}{32}|\eta| \}\\
&\subset\{(\xi,\eta)\in\R^2\times\R^2\,|\,\sqrt\epsilon|\eta|\leq 64,\,\f{31}{32}|\eta|\leq|\xi|\leq\f{33}{32}|\eta|\}.
\end{aligned}\eeq

For $q_{\mu,+}^\epsilon(\xi,\eta)$, we only consider $\angle(\xi,\eta)\neq0$. Thanks to Sine theorem, we have
\beno
\f{|\xi-\eta|}{\sin\bigl(\angle(\xi,\eta)\bigr)}=\f{|\eta|}{\sin\bigl(\angle(\xi,\xi-\eta)\bigr)}
=\f{|\xi|}{\sin\bigl(\angle(\xi-\eta,\eta)\bigr)}.
\eeno
By virtue of \eqref{supp of q}, we see that $\angle(\xi,\eta)<\f{\pi}{2}$ so that $\sin\bigl(\angle(\xi,\eta)\bigr)\sim\angle(\xi,\eta)$. Then we have
\beno
|\xi|\angle(\xi,\eta)\sim |\xi-\eta|\sin\bigl(\angle(\xi-\eta,\eta)\bigr),
\eeno
which along with \eqref{B3} implies
\beno
|q_{\mu,+}^\epsilon(\xi,\eta)|\lesssim\epsilon|\xi-\eta|\cdot\varphi_{\leq5}(\sqrt\epsilon|\eta|)\varphi_{\leq-6}\Bigl(\f{|\xi-\eta|}{|\eta|}\Bigr).
\eeno
This is the second part of \eqref{bound of q eps}.

\medskip

(3) {\it For the cubic term $\mathcal{C}_V^\epsilon$,} we first have
\beno\begin{aligned}
&\widehat{\mathcal{C}_V^\epsilon}(\xi)=\f{i\epsilon^2}{4\pi^2}\int_{\R^2}\bigl(\xi\cdot\f{\xi-\eta}{|\xi-\eta|^2}\bigr)\cdot
\bigl((\xi-\eta)\cdot\widehat{\vv B^\epsilon}(\xi-\eta)\bigr)\sum_{j\in\Z}\varphi_{\leq j-7}(|\xi-\eta|)\varphi_j(|\eta|)\widehat{V}(\eta)d\eta.
\end{aligned}\eeno
Then there holds
\beno\begin{aligned}
&\bigl(\langle\na\rangle^{N_0}\mathcal{C}_V^\epsilon\,|\,\langle\na\rangle^{N_0}V\bigr)_2
=\f{i\epsilon^2}{(2\pi)^4}\int_{\R^4}\langle\xi\rangle^{2N_0}\bigl(\xi\cdot\f{\xi-\eta}{|\xi-\eta|^2}\bigr)\cdot
\bigl((\xi-\eta)\cdot\widehat{\vv B^\epsilon}(\xi-\eta)\bigr)\\
&\qquad\cdot\sum_{j\in\Z}\varphi_{\leq j-7}(|\xi-\eta|)\varphi_j(|\eta|)\widehat{V}(\eta)\overline{\widehat{V}(\xi)}d\eta d\xi
\end{aligned}\eeno
and
\beno\begin{aligned}
&\overline{\bigl(\langle\na\rangle^{N_0}\mathcal{C}_V^\epsilon\,|\,\langle\na\rangle^{N_0}V\bigr)_2}\\
&=-\f{i\epsilon^2}{(2\pi)^4}\int_{\R^4}\langle\xi\rangle^{2N_0}\bigl(\xi\cdot\f{\xi-\eta}{|\xi-\eta|^2}\bigr)\cdot
\bigl((\xi-\eta)\cdot\overline{\widehat{\vv B^\epsilon}(\xi-\eta)}\bigr)\sum_{j\in\Z}\varphi_{\leq j-7}(|\xi-\eta|)\varphi_j(|\eta|)\widehat{V}(\xi)\overline{\widehat{V}(\eta)}d\eta d\xi\\
&=-\f{i\epsilon^2}{(2\pi)^4}\int_{\R^4}\langle\eta\rangle^{2N_0}\bigl(\eta\cdot\f{\xi-\eta}{|\xi-\eta|^2}\bigr)\cdot
\bigl((\xi-\eta)\cdot\overline{\widehat{\vv B^\epsilon}(\eta-\xi)}\bigr)\sum_{j\in\Z}\varphi_{\leq j-7}(|\xi-\eta|)\varphi_j(|\xi|)\widehat{V}(\eta)\overline{\widehat{V}(\xi)}d\eta d\xi.
\end{aligned}\eeno

Thanks to \eqref{estimate for Beps}, we have
\beno
\overline{\widehat{\vv B^\epsilon}(\eta-\xi)}=\widehat{\vv B^\epsilon}(\xi-\eta)
\eeno
which leads to
\beno\begin{aligned}
&\text{Re}\bigl\{\bigl(\langle\na\rangle^{N_0}\mathcal{C}_V^\epsilon\,|\,\langle\na\rangle^{N_0}V\bigr)_2\}
=\f{i\epsilon^2}{2(2\pi)^4}\int_{\R^4}\bigl({\vv c}^\epsilon(\xi,\eta)\cdot\widehat{\vv B^\epsilon}(\xi-\eta)\bigr)\cdot\langle\eta\rangle^{N_0}\widehat{V}(\eta)\cdot\overline{\langle\xi\rangle^{N_0}\widehat{V}(\xi)}d\eta d\xi,
\end{aligned}\eeno
where
\beno\begin{aligned}
{\vv c}^\epsilon(\xi,\eta)&=\langle\xi\rangle^{-N_0}\langle\eta\rangle^{-N_0}\sum_{j\in\Z}\varphi_{\leq j-7}(|\xi-\eta|)\Bigl(\langle\xi\rangle^{2N_0}\bigl(\xi\cdot(\xi-\eta)\bigr)\varphi_j(|\eta|)\\
&\qquad
-\langle\eta\rangle^{2N_0}\bigl(\eta\cdot(\xi-\eta)\bigr)\varphi_j(|\xi|)\Bigr)\f{\xi-\eta}{|\xi-\eta|^2}.
\end{aligned}\eeno

Since for fixed $(\xi,\eta)$,
\beno
|\varphi_j(|\xi|)-\varphi_j(|\eta|)|\lesssim\f{1}{\min\{|\xi|,|\eta|\}}|\xi-\eta|,
\eeno
and the summation in ${\vv c}^\epsilon(\xi,\eta)$ is finite, using \eqref{B2}, we have $|\xi|\sim|\eta|$ and
\beno
|{\vv c}^\epsilon(\xi,\eta)|\lesssim|\xi-\eta|.
\eeno
Then we obtain
\beno\begin{aligned}
&\bigl|\text{Re}\bigl\{\bigl(\langle\na\rangle^{N_0}\mathcal{C}_V^\epsilon\,|\,\langle\na\rangle^{N_0}V\bigr)_2\}\bigr|
\lesssim\epsilon^2\int_{\R^4}|\xi-\eta||\widehat{\vv B^\epsilon}(\xi-\eta)|\cdot\langle\eta\rangle^{N_0}|\widehat{V}(\eta)|
\cdot\langle\xi\rangle^{N_0}|\widehat{V}(\xi)|d\eta d\xi\\
&\lesssim\epsilon^2\|\xi\widehat{\vv B^\epsilon}(\xi)\|_{L^1}\|\langle\xi\rangle^{N_0}\widehat{V}(\xi)\|_{L^2}^2
\lesssim\epsilon^2\|\langle\xi\rangle^3\widehat{\vv B^\epsilon}(\xi)\|_{L^2}\|V\|_{H^{N_0}}^2\lesssim\epsilon^2\|\vv B^\epsilon\|_{H^3}\|V\|_{H^{N_0}}^2,
\end{aligned}\eeno
which along with \eqref{estimate for Beps a} implies
\beno
\bigl|\text{Re}\bigl\{\bigl(\langle\na\rangle^{N_0}\mathcal{C}_V^\epsilon\,|\,\langle\na\rangle^{N_0}V\bigr)_2\}\bigr|
\lesssim\epsilon^2\|\z\|_{L^\infty}\|\vv v\|_{H^3}\|V\|_{H^{N_0}}^2\lesssim\epsilon^2\|\z\|_{H^{N_0}}\|\vv v\|_{H^{N_0}}\|V\|_{H^{N_0}}^2.
\eeno
This is \eqref{estimate for cubic term eps}.

\medskip

(4) {\it For the nonlinear term $\mathcal{N}_V^\epsilon$,}   we estimate it term by term.

Firstly, by definition of para-differential operators and \eqref{estimate for Beps a}, we have
\beno\begin{aligned}
&\epsilon^2\|\na\cdot\bigl(T_\z\varphi_{\leq5}(\sqrt\epsilon|D|)\D^{-1}\na\div\vv B^\epsilon\bigr)\|_{H^{N_0}}\\
&
\lesssim\epsilon^2\|\z\|_{L^\infty}\|\vv B^\epsilon\|_{H^{N_0+1}}
\lesssim\epsilon^{\f32}\|\z\|_{L^\infty}^2\|\vv v\|_{H^{N_0}}\lesssim\epsilon^{\f32}\|\z\|_{H^{N_0}}^2\|\vv v\|_{H^{N_0}}.
\end{aligned}\eeno
Using \eqref{commutator}, we have
\beno\begin{aligned}
&\epsilon\|[|D|^{-1}\div,T_{\vv v}]\cdot\na\vv u\|_{H^{N_0}}\lesssim\epsilon\|\na\vv v\|_{L^\infty}\|\vv u\|_{H^{N_0}}
\lesssim\epsilon\|\vv v\|_{H^{N_0}}\|\vv u\|_{H^{N_0}},\\
&\f{\epsilon}{2}\||D|\bigl([|D|^{-1}\div,\,T_\z]\vv u\bigr)\|_{H^{N_0}}\lesssim\epsilon\|\na\z\|_{L^\infty}\|\vv u\|_{H^{N_0}}
\lesssim\epsilon\|\z\|_{H^{N_0}}\|\vv u\|_{H^{N_0}}.
\end{aligned}\eeno
By the definition of para-differential operators, we have
\beno
\epsilon \|T_{\div\vv v}(|D|^{-1}\div\vv u)\|_{H^{N_0}}\lesssim\epsilon\|\na\vv v\|_{L^\infty}\|\vv u\|_{H^{N_0}}
\lesssim\epsilon\|\vv v\|_{H^{N_0}}\|\vv u\|_{H^{N_0}}.
\eeno
Then we obtain
\beq\label{B4}
\|\mathcal{N}_V^\epsilon\|_{H^{N_0}}\lesssim\epsilon\bigl(\|\vv v\|_{H^{N_0}}+\|\z\|_{H^{N_0}}\bigr)\|\vv u\|_{H^{N_0}}
+\epsilon\|\z\|_{H^{N_0}}^2\|\vv v\|_{H^{N_0}}
+\|N_\z^\epsilon\|_{H^{N_0}}+\|N_{\vv u}^\epsilon\|_{H^{N_0}}.
\eeq

For $N_\z^\epsilon$, due to the definition of para-differential operators and \eqref{estimate for Beps a}, we have
\beno\begin{aligned}
&\|\epsilon\na\cdot\bigl(R(\z,\vv v)\bigr)+\f{\epsilon^2}{2}\na\cdot\bigl(T_\z \vv B^\epsilon\bigr)\|_{H^{N_0}}
\lesssim\epsilon\|\na\z\|_{L^\infty}\|\vv v\|_{H^{N_0}}+\epsilon^2\|\z\|_{L^\infty}\|\vv B^\epsilon(\z,\vv v)\|_{H^{N_0+1}}\\
&\quad\lesssim\epsilon\|\na\z\|_{L^\infty}\|\vv v\|_{H^{N_0}}+\epsilon^{\f32}\|\z\|_{L^\infty}^2\|\vv v\|_{H^{N_0}}
\lesssim\epsilon\bigl(\|\z\|_{H^{N_0}}+\|\z\|_{H^{N_0}}^2\bigr)\|\vv v\|_{H^{N_0}},
\end{aligned}\eeno
and
\beno\begin{aligned}
&\f{\epsilon^2}{2}\|\na\cdot\bigl([\D,T_\z](1+\epsilon\D)^{-1}\varphi_{\geq6}(\sqrt\epsilon|D|)\vv v\bigr)\|_{H^{N_0}}\lesssim\epsilon\|\na\z\|_{L^\infty}\|\vv v\|_{H^{N_0}}\lesssim\epsilon\|\z\|_{H^{N_0}}\|\vv v\|_{H^{N_0}},
\end{aligned}\eeno
which imply
\beq\label{B5}
\|N_\z^\epsilon\|_{H^{N_0}}\lesssim\epsilon\bigl(1+\|\z\|_{H^{N_0}}\bigr)\|\z\|_{H^{N_0}}\|\vv v\|_{H^{N_0}}.
\eeq

For $N_{\vv u}^\epsilon$, by virtue of  the definition of para-differential operators, we have
\beno\begin{aligned}
&\|\f{\epsilon}{2}T_{\na\z}\varphi_{\geq6}(\sqrt\epsilon|D|)\z-\epsilon T_{\na\vv v}\cdot\vv v-\epsilon R(\vv v\cdot,\na\vv v)\|_{H^{N_0}}\\
&\quad\lesssim\epsilon\bigl(\|\na\z\|_{L^\infty}\|\z\|_{H^{N_0}}+\|\na\vv v\|_{L^\infty}\|\vv v\|_{H^{N_0}}\bigr)
\lesssim\epsilon\bigl(\|\z\|_{H^{N_0}}^2+\|\vv v\|_{H^{N_0}}^2\bigr),
\end{aligned}\eeno
By the definition of $\vv B^\epsilon(\cdot,\cdot)$, we have
\beno\begin{aligned}
&\epsilon\|\vv B^\epsilon\bigl((1+\epsilon\D)\na\cdot\vv v,\vv v\bigr)\|_{H^{N_0}}
=\f{\epsilon}{2}\|T_{(1+\epsilon\D)\na\cdot\vv v}\bigl((1+\epsilon\D)^{-1}\varphi_{\geq6}(\sqrt\epsilon|D|)\vv v\bigr)\|_{H^{N_0}}\\
&\quad
\lesssim\epsilon\|\na\vv v\|_{L^\infty}\|\vv v\|_{H^{N_0}}\lesssim\epsilon\|\vv v\|_{H^{N_0}}^2.
\end{aligned}\eeno

Using the definition of $\vv B^\epsilon(\cdot,\cdot)$ again and  \eqref{estimate for Beps a}, we have
\beno\begin{aligned}
&\epsilon^2\|T_{\vv v}\cdot\na\vv B^\epsilon\|_{H^{N_0}}\lesssim\epsilon^2\|\vv v\|_{L^\infty}\|\vv B^\epsilon(\z,\vv v)\|_{H^{N_0+1}}
\lesssim\epsilon^{\f32}\|\vv v\|_{L^\infty}\|\z\|_{L^\infty}\|\vv v\|_{H^{N_0}}
\lesssim\epsilon^{\f32}\|\z\|_{H^{N_0}}\|\vv v\|_{H^{N_0}}^2,\\
&\epsilon^2\|\vv B^\epsilon\bigl(\na\cdot(\z\vv v),\vv v\bigr)\|_{H^{N_0}}\lesssim\epsilon^2\|\na\cdot(\z\vv v)\|_{L^\infty}\|\vv v\|_{H^{N_0}}
\lesssim\epsilon^2\|\z\|_{H^{N_0}}\|\vv v\|_{H^{N_0}}^2,\\
&\f{\epsilon^2}{2}\|\vv B^\epsilon\bigl(\z,\na(|\vv v|^2)\bigr)\|_{H^{N_0}}\lesssim\epsilon^{\f32}\|\z\|_{L^\infty}\|\na(|\vv v|^2)\|_{H^{N_0-1}}
\lesssim\epsilon^{\f32}\|\z\|_{H^{N_0}}\|\vv v\|_{H^{N_0}}^2.
\end{aligned}\eeno
Then we obtain
\beq\label{B6}
\|N_{\vv u}^\epsilon\|_{H^{N_0}}\lesssim\epsilon\|\z\|_{H^{N_0}}^2+\epsilon\bigl(1+\|\z\|_{H^{N_0}}\bigr)\|\vv v\|_{H^{N_0}}^2.
\eeq

Combining \eqref{B4}, \eqref{B5} and \eqref{B6}, we have
\beno
\|\mathcal{N}_V^\epsilon\|_{H^{N_0}}\lesssim\epsilon\bigl(1+\|\z\|_{H^{N_0}}\bigr)\bigl(\|\z\|_{H^{N_0}}^2+\|\vv v\|_{H^{N_0}}^2+\|\vv u\|_{H^{N_0}}^2\bigr).
\eeno
This is exactly \eqref{estimate for remain terms}.

Therefore, we complete the proof of the proposition.
\end{proof}

\setcounter{equation}{0}
\section{Proof of Theorem \ref{main theorem}}

In this section, we shall prove the main Theorem \ref{main theorem} via the symmetric formulation \eqref{New Bsq for 2D} and the normal form techniques. The proof relies on the standard continuity argument and the main priori estimates for \eqref{Bsq2D} which are stated in the following subsections.

\subsection{Ansatz  for the continuity arguments}
The first ansatz for the continuity arguments involves the amplitude $\z$ as follows
\beq\label{ansatz 1}
\epsilon\|\z\|_{L^\infty}\leq\f{1}{2C_B},\quad \text{for any}\, t\in[0,T_0\epsilon^{-\f23}],
\eeq
where $C_B$ is a constant stated in Lemma \ref{lem for Beps}.

Let us define the energy functional for \eqref{Bsq2D} as
\beno
\mathcal{E}_{N_0}(t)\eqdefa\|\z(t,\cdot)\|_{H^{N_0}}^2+\|\vv v(t,\cdot)\|_{H^{N_0}}^2.
\eeno
For simplicity and without loss of generality, we assume that
\beq\label{initial energy}
\mathcal{E}_{N_0}(0)\eqdefa\|\z_0\|_{H^{N_0}}^2+\|\vv v_0\|_{H^{N_0}}^2=1.
\eeq

The second ansatz is about the energy which reads
\beq\label{ansatz 2}
\mathcal{E}_{N_0}(t)\leq 2C_0,\quad \text{for any}\, t\in[0,T_0\epsilon^{-\f23}],
\eeq
where $C_0>1$ is an universal constant that will be determined at the end of the proof. We take
\beno
T_0=\f{C_1}{C_2},\quad C_0=2C_1,
\eeno
where $C_1$ and $C_2$ are the constants stated in the following Proposition \ref{priori energy estimate prop}.

The standard continuity argument shows that: since for sufficiently small $\epsilon>0$,
\beno
\mathcal{E}_{N_0}(0)=1\leq 2C_0,\quad\epsilon\|\z_0\|_{L^\infty}\leq\f{1}{4C_B}\leq\f{1}{2C_B},
\eeno
ansatz \eqref{ansatz 1} and \eqref{ansatz 2} hold on a short time interval $[0,t^*)$, where $t^*$ is the maximal possible time on which \eqref{ansatz 1} and \eqref{ansatz 2} are correct. Without loss of generality, we assume that $t^*=T_0\epsilon^{-\f23}$.

To close the continuity argument, we need to verify that: {\it there exists a sufficiently small $\epsilon_0>0$ such that for any $\epsilon\in(0,\epsilon_0)$, we improve the ansatz \eqref{ansatz 1} and \eqref{ansatz 2} as follows
\beq\label{ansatz 1a}
\epsilon\|\z\|_{L^\infty}\leq\f{1}{4C_B},\quad \text{for any}\, t\in[0,T_0\epsilon^{-\f23}],
\eeq
\beq\label{ansatz 2a}
\mathcal{E}_{N_0}(t)\leq C_0,\quad \text{for any}\, t\in[0,T_0\epsilon^{-\f23}].
\eeq
}

\medskip

Theorem \ref{main theorem} follows from the above continuity argument and the local regularity theorem. To complete the last step of the continuity argument that improves the constants in the ansatz, we need the following   Proposition \ref{priori energy estimate prop} that concerns the {\it a priori} energy estimates on \eqref{Bsq2D}.

\subsection{The {\it a priori} energy estimates.}
In this subsection, we shall derive the {\it a priori} energy estimates on the solutions of \eqref{Bsq2D}-\eqref{initial}. The main result is stated in the following proposition.

\begin{proposition}\label{priori energy estimate prop}
Under the ansatz \eqref{ansatz 1} and \eqref{ansatz 2}, the solution $(\z,\vv v)$ to \eqref{Bsq2D}-\eqref{initial} satisfies
\beq\label{priori energy}
\mathcal{E}_{N_0}(t)\leq C_1+C_2\epsilon^{\f23}t,\quad\text{for any }\, t\in(0,T_0\epsilon^{-\f23}],
\eeq
where $C_1$ and $C_2$ are universal constants, and $T_0=\f{C_1}{C_2}$.
\end{proposition}
\begin{proof}
We divide the proof into several steps.

\medskip

{\bf Step 1. Energy estimates.}  Firstly, thanks to \eqref{estimate for Beps a} and \eqref{ansatz 1}, we have
\beno
\epsilon\|\vv B^\epsilon(\z,\vv v)\|_{H^{N_0}}\leq C_B\epsilon\|\z\|_{L^\infty}\|\vv v\|_{H^{N_0}}\leq\f{1}{2}\|\vv v\|_{H^{N_0}},
\eeno
which along with \eqref{good unknowns for 2D} and \eqref{good variable for 2D} implies
\beq\label{equivalent energy 1}
\mathcal{E}_{N_0}(t)\sim \|\z\|_{H^{N_0}}^2+\|\vv u\|_{H^{N_0}}^2\sim\|V\|_{H^{N_0}}^2.
\eeq

With \eqref{equivalent energy 1}, we start to derive the energy estimates on the dispersive formulation \eqref{symmetric formula} as follows
\beno\begin{aligned}
&\f12\f{d}{dt}\|V(t)\|_{H^{N_0}}^2=\text{Re}\{\bigl(\langle \na\rangle^{N_0}\mathcal{S}^\epsilon_V\,|\,\langle\na\rangle^{N_0}V\bigr)_2
+\bigl(\langle\na\rangle^{N_0}\mathcal{Q}^\epsilon_V\,|\,\langle\na\rangle^{N_0}V\bigr)_2\\
&\qquad
+\bigl(\langle\na\rangle^{N_0}\mathcal{C}^\epsilon_V\,|\,\langle\na\rangle^{N_0}V\bigr)_2
+\bigl(\langle\na\rangle^{N_0}\mathcal{N}^\epsilon_V\,|\,\langle\na\rangle^{N_0}V\bigr)_2\}.
\end{aligned}\eeno

Thanks to the estimates \eqref{estimate for cubic term eps} and \eqref{estimate for remain terms} in Proposition \ref{New Bsq 2D prop}, using \eqref{initial energy}, \eqref{ansatz 2} and \eqref{equivalent energy 1}, we obtain
\beq\label{priori for Bsq}
\mathcal{E}_{N_0}(t)\lesssim 1+|\text{Re}(I)|+|\text{Re}(II)|+|\text{Re}(III)|+\epsilon t,
\eeq
where
\beq\label{quadratic  terms e}\begin{aligned}
&I\eqdefa\sum_{\mu\in\{+,-\}}\int_0^t\int_{\R^2\times\R^2}\langle\xi\rangle^{2N_0}s^\epsilon_{\mu,+}(\xi,\eta)\widehat{V^\mu}(\xi-\eta)\widehat{V^+}(\eta)
\overline{\widehat{V^+}(\xi)}d\eta d\xi dt,\\
&II\eqdefa\sum_{\mu\in\{+,-\}}\int_0^t\int_{\R^2\times\R^2}\langle\xi\rangle^{2N_0}q^\epsilon_{\mu,+}(\xi,\eta)\widehat{V^\mu}(\xi-\eta)\widehat{V^+}(\eta)
\overline{\widehat{V^+}(\xi)}d\eta d\xi dt,\\
&III\eqdefa\sum_{\mu\in\{+,-\}}\int_0^t\int_{\R^2\times\R^2}\langle\xi\rangle^{2N_0}q^\epsilon_{\mu,-}(\xi,\eta)\widehat{V^\mu}(\xi-\eta)\widehat{V^-}(\eta)
\overline{\widehat{V^+}(\xi)}d\eta d\xi dt.
\end{aligned}\eeq

By virtue of \eqref{bound of s eps} and \eqref{bound of q eps}, it is easy to get the energy estimate
\beno
\mathcal{E}_{N_0}(t)\lesssim 1+\sqrt\e t,
\eeno
which gives rise to the local existence of \eqref{Bsq2D}-\eqref{initial} on a finite time interval of scalar $\f{1}{\sqrt\epsilon}$. To enlarge the time scalar $\f{1}{\sqrt\epsilon}$, we shall use normal formal techniques to deal with the worst term $III$ that is involving the quadratic term $Q_{\mu,-}^\epsilon(V^\mu,V^-)$.

\medskip

{\bf Step 2. Estimate for $\text{Re}(I)$ and $\text{Re}(II)$.} In this step, we shall derive the following estimate
\beq\label{estimate for I and II}
|\text{Re}(I)|+|\text{Re}(II)|\lesssim\epsilon t.
\eeq

For $I$ , by the expression of $I$, we have
\beno
\bar{I}=\sum_{\mu\in\{+,-\}}\int_0^t\int_{\R^2\times\R^2}\langle\xi\rangle^{2N_0}\overline{s^\epsilon_{\mu,+}(\xi,\eta)}
\overline{\widehat{V^\mu}(\xi-\eta)}\overline{\widehat{V^+}(\eta)}\widehat{V^+}(\xi)d\eta d\xi dt,
\eeno
which along with \eqref{symbol s eps} and the fact that
\beno
\overline{\widehat{V^\mu}(\xi)}=\widehat{V^{-\mu}}(-\xi)
\eeno
shows
\beno\begin{aligned}
\bar{I}&=-\sum_{\mu\in\{+,-\}}\int_0^t\int_{\R^2\times\R^2}\langle\xi\rangle^{2N_0}s^\epsilon_{\mu,+}(\xi,\eta)
\widehat{V^{-\mu}}(\eta-\xi)\overline{\widehat{V^+}(\eta)}\widehat{V^+}(\xi)d\eta d\xi dt\\
&=-\sum_{\mu\in\{+,-\}}\int_0^t\int_{\R^2\times\R^2}\langle\eta\rangle^{2N_0}s^\epsilon_{\mu,+}(\eta,\xi)
\widehat{V^{-\mu}}(\xi-\eta)\widehat{V^+}(\eta)\overline{\widehat{V^+}(\xi)}d\eta d\xi dt.
\end{aligned}\eeno
Since $\text{Re}(I)=\f12(I+\bar{I})$, we have
\beno
\text{Re}(I)=\sum_{\mu\in\{+,-\}}\int_0^t\int_{\R^2\times\R^2}\tilde{s}^\epsilon_{\mu,+}(\xi,\eta)
\widehat{V^\mu}(\xi-\eta)\cdot\langle\eta\rangle^{N_0}\widehat{V^+}(\eta)\cdot\langle\xi\rangle^{N_0}
\overline{\widehat{V^+}(\xi)}d\eta d\xi dt,
\eeno
where
\beno
\tilde{s}^\epsilon_{\mu,+}(\xi,\eta)\eqdefa\f12\langle\xi\rangle^{-N_0}\langle\eta\rangle^{-N_0}\bigl(\langle\xi\rangle^{2N_0}s_{\mu,+}^\epsilon(\xi,\eta)
-\langle\eta\rangle^{2N_0}s_{-\mu,+}^\epsilon(\eta,\xi)\bigr).
\eeno

Thanks to \eqref{bound of s eps}, we have
\beno
|\tilde{s}^\epsilon_{\mu,+}(\xi,\eta)|\lesssim\epsilon|\xi-\eta|\cdot\varphi_{\leq -6}\Bigl(\f{|\xi-\eta|}{\max\{|\xi|,|\eta|\}}\Bigr).
\eeno
Then we have
\beno\begin{aligned}
|\text{Re}(I)|&\lesssim\epsilon\sum_{\mu\in\{+,-\}}\int_0^t\int_{\R^2\times\R^2}|\xi-\eta||\widehat{V^\mu}(\xi-\eta)|
\cdot\langle\eta\rangle^{N_0}|\widehat{V^+}(\eta)|\cdot\langle\xi\rangle^{N_0}|\widehat{V^+}(\xi)|d\eta d\xi dt\\
&\lesssim\epsilon t\sup_{\tau\in(0,t)}\bigl(\||\xi|\widehat{V^\mu}(\tau,\xi)\|_{L^1}\|\langle\xi\rangle^{N_0}\widehat{V^+}(\tau,\xi)\|_{L^2}^2\bigr)\\
&
\lesssim\epsilon t\sup_{\tau\in(0,t)}\bigl(\|V^\mu\|_{H^3}
\|V^+\|_{H^{N_0}}^2\bigr)
\lesssim\epsilon t\sup_{\tau\in(0,t)}\|V(\tau)\|_{H^{N_0}}^3.
\end{aligned}\eeno
By virtue of \eqref{equivalent energy 1} and \eqref{ansatz 2}, we obtain
\beq\label{estimate for I}
|\text{Re}(I)|\lesssim \epsilon t.
\eeq

Following a similar derivation as for \eqref{estimate for I}, using the second inequality of \eqref{bound of q eps}, we have
\beq\label{estimate for II}
|\text{Re}(II)|\lesssim|II|\lesssim \epsilon t.
\eeq

Due to \eqref{estimate for I} and \eqref{estimate for II}, we obtain \eqref{estimate for I and II}.

\medskip

{\bf Step 3. Estimate for $\text{Re}(III)$.} Due to the first inequality of \eqref{bound of q eps},  $III$ is of order $O(\sqrt\epsilon)$. In order to improve the bound of $III$, we shall apply a normal form technique to $III$. The aim of this step is to derive the following estimate
\beq\label{estimate for III}
|III|\lesssim 1+\epsilon^{\f23} t.
\eeq

{\it Step 3.1. The evolution equation and estimates of the profile.} We introduce the profiles $f$ and $g$ of $V$ and $\langle\na\rangle^{N_0}V$ as follows
\beno
f=e^{-it\Lambda_\epsilon}V\quad\text{and}\quad g=\langle\na\rangle^{N_0}f.
\eeno
Due to \eqref{equivalent energy 1}, we have
\beq\label{equivalent energy 2}
\mathcal{E}_{N_0}(t)\sim\|V\|_{H^{N_0}}^2\sim\|f\|_{H^{N_0}}^2=\|g\|_{L^2}^2.
\eeq

By virtue of \eqref{linearized system}, we have
\beq\label{evolution equation for profile}\begin{aligned}
&\p_tf=\epsilon e^{-it\Lambda_\epsilon}\bigl\{- \na\cdot(\z\vv v)+(1+\epsilon\D)\na\cdot\vv B^\epsilon\\
&\qquad
+i|D|^{-1}\div\bigl[-\f{1}{2}\na(|\vv v|^2)+\vv B^\epsilon(\p_t\z,\vv v)+\vv B^\epsilon(\z,\p_t\vv v)\bigr]\bigr\}
\end{aligned}\eeq
Then we have
\beno\begin{aligned}
\||D|^{-1}\p_tf\|_{H^{N_0}}&\lesssim\epsilon\bigl(\|\z\vv v\|_{H^{N_0}}+\|(1+\epsilon\D)\vv B^\epsilon\|_{H^{N_0}}
+\||\vv v|^2\|_{H^{N_0}}\\
&\qquad+\||D|^{-1}\bigl(\vv B^\epsilon(\p_t\z,\vv v)\bigr)\|_{H^{N_0}}
+\||D|^{-1}\bigl(\vv B^\epsilon(\z,\p_t\vv v)\bigr)\|_{H^{N_0}}\bigr).
\end{aligned}\eeno
By virtue of the definition of $\vv B^\epsilon(\cdot,\cdot)$,  we have
\beno
\text{supp}\, \widehat{\vv B^\epsilon(\cdot,\cdot)}(\xi)\subset\{\xi\in\R^2\,|\, \sqrt\epsilon|\xi|\geq 2^5\},
\eeno
which along with \eqref{Bsq2D}, \eqref{estimate for Beps a} and the definition of $\vv B^\epsilon(\cdot,\cdot)$ implies
\beno\begin{aligned}
&\|(1+\epsilon\D)\vv B^\epsilon\|_{H^{N_0}}\lesssim\|\z\|_{L^\infty}\|\vv v\|_{H^{N_0}},\\
&\||D|^{-1}\bigl(\vv B^\epsilon(\p_t\z,\vv v)\bigr)\|_{H^{N_0}}\lesssim\sqrt\epsilon\|\p_t\z\|_{L^\infty}\|\vv v\|_{H^{N_0}}\\
&\quad
\lesssim\sqrt\epsilon\bigl(\|\vv v\|_{H^{5}}+\epsilon\|\z\|_{H^3}\|\vv v\|_{H^3}\bigr)\|\vv v\|_{H^{N_0}}
\lesssim\bigl(1+\epsilon\|\z\|_{H^{N_0}}\bigr)\|\vv v\|_{H^{N_0}}^2,\\
&\||D|^{-1}\bigl(\vv B^\epsilon(\z,\p_t\vv v)\bigr)\|_{H^{N_0}}\lesssim\|\z\|_{L^\infty}\||D|^{-1}\p_t\vv v\|_{H^{N_0}}\\
&\quad
\lesssim\|\z\|_{H^2}\bigl(\|\z\|_{H^{N_0}}+\epsilon\|\vv v\|_{H^{N_0}}^2\bigr)
\lesssim\|\z\|_{H^{N_0}}^2+\epsilon\|\z\|_{H^{N_0}}\|\vv v\|_{H^{N_0}}^2\bigr).
\end{aligned}\eeno

Thanks to \eqref{ansatz 2}, we have
\beq\label{estimate for profile}
\||D|^{-1}\p_tf\|_{H^{N_0}}\lesssim\epsilon.
\eeq

{\it Step 3.2. The profile formulation of $III$.} We denote
\beno
\mathfrak{Q}^\epsilon_{\mu}(\xi,\eta)=\langle\xi\rangle^{2N_0}q^\epsilon_{\mu,-}(\xi,\eta)\widehat{V^\mu}(\xi-\eta)\widehat{V^-}(\eta)
\overline{\widehat{V^+}(\xi)}.
\eeno

Now we rewrite $\mathfrak{Q}^\epsilon_{\mu}(\xi,\eta)$ in terms of the profiles $f$ and $g$ as follows
\beno
\mathfrak{Q}^\epsilon_{\mu}(\xi,\eta)=e^{it\Phi^\epsilon_{\mu,-}(\xi,\eta)}\tilde{q}^\epsilon_{\mu,-}(\xi,\eta)\widehat{f^{\mu}}(\xi-\eta)\cdot
\widehat{g^-}(\eta)\cdot\widehat{g^-}(-\xi),
\eeno
where
\beno\begin{aligned}
&\Phi^\epsilon_{\mu,-}(\xi,\eta)=-\Lambda_\epsilon(\xi)+\mu\Lambda_\epsilon(\xi-\eta)-\Lambda_\epsilon(\eta),\quad\tilde{q}^\epsilon_{\mu,-}(\xi,\eta)=\langle\eta\rangle^{-N_0} \langle\xi\rangle^{N_0} q^\epsilon_{\mu,-}(\xi,\eta).
\end{aligned}\eeno
Thanks to \eqref{bound of q eps}, we have $\text{supp}\,(\tilde{q}^\epsilon_{\mu,-})=\text{supp}\,(q^\epsilon_{\mu,-})$ and
\beq\label{bound of tilde q eps}\begin{aligned}
&|\tilde{q}^\epsilon_{\mu,-}(\xi,\eta)|\lesssim\epsilon|\xi|\cdot\varphi_{\leq 5}\bigl(\sqrt\epsilon|\eta|\bigr)\cdot\varphi_{\leq -6}\Bigl(\f{|\xi-\eta|}{|\eta|}\Bigr).
\end{aligned}\eeq

Due to Lemma \ref{phase lemma} and \eqref{supp of q}, for any $(\xi,\eta)\in\text{supp}\,(q^\epsilon_{\mu,-})$, we have
\beq\label{phase mu -}
\Phi^\epsilon_{\mu,-}(\xi,\eta)\sim|\eta|\phi^\epsilon_{\mu,-}(\xi,\eta),
\eeq
where $\phi^\epsilon_{+,-}(\xi,\eta)$  and $\phi^\epsilon_{-,-}(\xi,\eta)$ are defined in \eqref{phi + -} and \eqref{phi - -}.

{\it Step 3.3. Estimate for $\int_0^t\int_{\R^2\times\R^2}\mathfrak{Q}^\epsilon_{+}(\xi,\eta)d\eta d\xi d\tau$.}
 Firstly, due to \eqref{supp of q}, we have  for any  $(\xi,\eta)\in\text{supp}\,(q^\epsilon_{+,-})$,
 \beno
 \sin\bigl(\angle(\xi,\eta)\bigr)\leq\f{|\xi-\eta|}{|\eta|}\leq\f{1}{32},
 \eeno
 which gives rise to
 \beq\label{angle between xi and eta}
\sin\bigl(\angle(\xi,\eta)\bigr) \approx \angle(\xi,\eta)\leq\f{1}{31},\quad1\approx\cos\bigl(\f12\angle(\xi,\eta)\bigr)\geq\f{31}{32}.
 \eeq

 We divide the integral regime into the following three parts:

 {\it (1) For low frequencies $\sqrt\epsilon|\eta|\leq\f12$,}
 using \eqref{phi + -}, \eqref{angle between xi and eta}, \eqref{bound of tilde q eps} and \eqref{phase mu -}, we have
\beno
\phi^\epsilon_{+,-}(\xi,\eta)\approx 4(\epsilon|\eta|^2-1),
\eeno
so that
\beq\label{B7}
|\phi^\epsilon_{+,-}(\xi,\eta)|\sim 1\quad\text{and}\quad
|\f{\tilde{q}^\epsilon_{+,-}(\xi,\eta)}{i\Phi^\epsilon_{+,-}(\xi,\eta)}|\lesssim\f{\epsilon}{|\phi^\epsilon_{+,-}(\xi,\eta)|}\lesssim \epsilon.
\eeq

Integrating by parts w.r.t t, we have
\beno\begin{aligned}
&\int_0^t\int_{\R^2\times\R^2}\mathfrak{Q}^\epsilon_{+}(\xi,\eta)\varphi_{\leq-2}(\sqrt\epsilon|\eta|)d\eta d\xi dt\\
&=\underbrace{\int_{\R^2\times\R^2}\f{\tilde{q}^\epsilon_{+,-}(\xi,\eta)}{i\Phi^\epsilon_{+,-}(\xi,\eta)}
e^{i\tau\Phi^\epsilon_{+,-}(\xi,\eta)}\widehat{f^+}(\tau,\xi-\eta)\cdot
\widehat{g^-}(\tau,\eta)\cdot\widehat{g^-}(\tau,-\xi)\varphi_{\leq-2}(\sqrt\epsilon|\eta|)d\eta d\xi}_{A_1}|_{\tau=0}^t\\
&\quad-\underbrace{\int_0^t\int_{\R^2\times\R^2}\f{\tilde{q}^\epsilon_{+,-}(\xi,\eta)}{i\Phi^\epsilon_{+,-}(\xi,\eta)}
e^{it\Phi^\epsilon_{+,-}(\xi,\eta)}\p_t\bigl(\widehat{f^+}(\xi-\eta)\cdot
\widehat{g^-}(\eta)\cdot\widehat{g^-}(-\xi)\bigr)\varphi_{\leq-2}(\sqrt\epsilon|\eta|)d\eta d\xi dt}_{A_2}.
\end{aligned}\eeno

For $A_1$, by virtue of \eqref{B7}, we have
\beno\begin{aligned}
&|A_1|\lesssim\epsilon\sup_{\tau\in[0,t]}\int_{\R^2\times\R^2}|\widehat{f^+}(\tau,\xi-\eta)|\cdot
|\widehat{g^-}(\tau,\eta)|\cdot|\widehat{g^-}(\tau,-\xi)|d\eta d\xi\\
&\quad\lesssim\epsilon\sup_{\tau\in[0,t]}\bigl(\|\widehat{f^+}(\tau,\cdot)\|_{L^1}\|\widehat{g^-}(\tau,\cdot)\|_{L^2}^2\bigr)
\lesssim\epsilon\sup_{\tau\in[0,t]}\bigl(\|f(\tau,\cdot)\|_{H^2}\|g(\tau,\cdot)\|_{L^2}^2\bigr),
\end{aligned}\eeno
which along with \eqref{ansatz 2} and \eqref{equivalent energy 2} implies
\beq\label{estimate for A 1}
|A_1|\lesssim\epsilon.
\eeq

For $A_2$, using \eqref{B7}, we have
\beno\begin{aligned}
&|A_2|\lesssim\epsilon t\sup_{\tau\in[0,t]}\int_{\text{supp}\,(q^\epsilon_{+,-})}\bigl(|\xi-\eta|\cdot|\f{1}{|\xi-\eta|}\p_t\widehat{f^+}(\tau,\xi-\eta)|\cdot
|\widehat{g^-}(\tau,\eta)|\cdot|\widehat{g^-}(\tau,-\xi)|\\
&\qquad+|\eta|\cdot|\widehat{f^+}(\tau,\xi-\eta)|\cdot
|\f{1}{|\eta|}\p_t\bigl(\widehat{g^-}(\tau,\eta)\widehat{g^-}(\tau,-\xi)\bigr)|\bigr)\cdot 1_{\sqrt\epsilon|\eta|\leq\f12}d\eta d\xi\\
&\quad\lesssim\sqrt\epsilon t\sup_{\tau\in[0,t]}\bigl(\|\f{1}{|\cdot|}\p_t\widehat{f^+}(\tau,\cdot)\|_{L^1}\|\widehat{g^-}(\tau,\cdot)\|_{L^2}^2
+\|\widehat{f^+}(\tau,\cdot)\|_{L^1}\|\widehat{g^-}(\tau,\cdot)\|_{L^2}\|\f{1}{|\cdot|}\p_t\widehat{g^-}(\tau,\cdot)\|_{L^2}\bigr)\\
&\quad\lesssim\sqrt\epsilon t\sup_{\tau\in[0,t]}\bigl(\|\f{1}{|\cdot|}\p_tf(\tau,\cdot)\|_{H^2}\|g(\tau,\cdot)\|_{L^2}^2
+\|f(\tau,\cdot)\|_{H^2}\|g(\tau,\cdot)\|_{L^2}\|\f{1}{|\cdot|}\p_tg(\tau,\cdot)\|_{L^2}\bigr),
\end{aligned}\eeno
where we used the facts that $\sqrt\epsilon|\eta|\leq\f{1}{2}$, $|\xi-\eta|\leq\f{1}{32}|\eta|$ and $|\xi|\sim|\eta|$ in the second inequality.
By virtue of \eqref{ansatz 2}, \eqref{equivalent energy 2} and \eqref{estimate for profile}, we obtain
\beno
|A_2|\lesssim\epsilon^{\f32} t,
\eeno
which along with \eqref{estimate for A 1} shows
\beq\label{estimate for Q + in low freq}
\bigl|\int_0^t\int_{R^2\times\R^2}\mathfrak{Q}^\epsilon_{+}(\xi,\eta)\varphi_{\leq-2}(\sqrt\epsilon|\eta|)d\eta d\xi dt\bigr|
\lesssim\epsilon+\epsilon^{\f32} t.
\eeq

{\it (2) For moderate frequencies with phase far away from the spatial resonance set,} i.e.,
\beno
\f14\leq\sqrt\epsilon|\eta|\leq64,\quad |\phi^\epsilon_{+,-}(\xi,\eta)|\geq 2^{-D-1},
\eeno
we have
\beno
|\f{\tilde{q}^\epsilon_{+,-}(\xi,\eta)}{i\Phi^\epsilon_{+,-}(\xi,\eta)}|\lesssim\f{\epsilon}{|\phi^\epsilon_{+,-}(\xi,\eta)|}\lesssim 2^D\epsilon.
\eeno
Here $D\in\N$ is a large number which will be determined later on.

Following similar a derivation as in \eqref{estimate for Q + in low freq}, we obtain
\beq\label{estimate for Q + in med freq with large modulation}
\bigl|\int_0^t\int_{\R^2\times\R^2}\mathfrak{Q}^\epsilon_{+}(\xi,\eta)\varphi_{[-1,5]}(\sqrt\epsilon|\eta|)
\varphi_{\geq-D}(\phi^\epsilon_{+,-}(\xi,\eta))d\eta d\xi dt\bigr|
\lesssim2^D\epsilon+2^D\epsilon^{\f32}t.
\eeq

{\it (3) For moderate frequencies with phase near the spatial resonance set,} i.e.,
\beno
\f14\leq\sqrt\epsilon|\eta|\leq64,\quad |\phi^\epsilon_{+,-}(\xi,\eta)|\leq 2^{-D},
\eeno
we shall split the integral regime into the following two parts
\beno
\f{|\xi-\eta|}{|\eta|}\leq 2^{-K+1}\quad\text{and}\quad\f{|\xi-\eta|}{|\eta|}\in[2^{-K},2^{-5}],
\eeno
where $K\in\N$ is a large number which will be determined later on. We use the smallness of $\angle(\xi,\eta)$ for the former part, while we use
the smallness of the symbol $\tilde{q}^\epsilon_{+,-}$ and the volume of the integral regime for the latter part.

{\it (i) For case $\f{|\xi-\eta|}{|\eta|}\leq 2^{-K+1}$,}
 using Sine theorem, we have
\beno
\sin\bigl(\angle(\xi,\eta)\bigl)=\sin\bigl(\angle(\xi-\eta,\xi)\bigr)\f{|\xi-\eta|}{|\eta|}\leq2^{-K+1},
\eeno
which along with \eqref{angle between xi and eta} gives rise to
\beq\label{angular 1}
2^l\sim\angle(\xi,\eta)\leq2^{-K+2},\quad\text{i.e., } l\leq-K+3.
\eeq

Now localizing the angle $\angle(\xi,\eta)$, we have
\beq\label{B15}\begin{aligned}
&\bigl|\int_0^t\int_{\R^2\times\R^2}\mathfrak{Q}^\epsilon_{+}(\xi,\eta)\varphi_{[-1,5]}(\sqrt\epsilon|\eta|)\cdot
\varphi_{\leq-D-1}(\phi^\epsilon_{+,-}(\xi,\eta))\cdot\varphi_{\leq-K}(\f{|\xi-\eta|}{|\eta|})\\
&\qquad\cdot\varphi_{k}(|\xi|)
\varphi_{k_1}(|\xi-\eta|)\varphi_{k_2}(|\eta|)\varphi_{l}(\angle(\xi,\eta))d\eta d\xi dt\bigr|\\
&\lesssim t\sup_{\tau\in[0,t]}\bigl(\|\widehat{g_k}\|_{L^2}\|T_k(f,g)\|_{L^2_\xi}\bigr),
\end{aligned}\eeq
where
\beno\begin{aligned}
T_k(f,g)&\eqdefa\int_{\R^2}e^{it\Phi^\epsilon_{+,-}(\xi,\eta)}\tilde{q}^\epsilon_{+,-}(\xi,\eta)\varphi_{[-1,5]}(\sqrt\epsilon|\eta|)\cdot
\varphi_{\leq-D-1}(\phi^\epsilon_{+,-}(\xi,\eta))\cdot\varphi_{\leq-K}(\f{|\xi-\eta|}{|\eta|})\\
&\qquad\cdot\widehat{f^{+}}(\xi-\eta)
\widehat{g^-}(\eta)\varphi_{k}(|\xi|)
\varphi_{k_1}(|\xi-\eta|)\varphi_{k_2}(|\eta|)\varphi_{l}(\angle(\xi,\eta))d\eta
\end{aligned}\eeno
Due to \eqref{bound of tilde q eps}, there holds $|\tilde{q}^\epsilon_{+,-}(\xi,\eta)|\lesssim\sqrt\epsilon$ and $|k-k_2|\leq2$.
 Then using the bilinear estimate
\eqref{biliear estimate}, we have
\beq\label{B8}
\|T_k(f,g)\|_{L^2_\xi}\lesssim \sqrt\epsilon2^{\f{l}{2}}\cdot2^{k_1}\|f_{k_1}\|_{L^2}\|g_{k_2}\|_{L^2}.
\eeq

Thanks to \eqref{angular 1}, \eqref{B15} and \eqref{B8}, we obtain
\beno\begin{aligned}
&\bigl|\int_0^t\int_{\R^2\times\R^2}\mathfrak{Q}^\epsilon_{+}(\xi,\eta)\varphi_{[-1,5]}(\sqrt\epsilon|\eta|)\cdot
\varphi_{\leq-D-1}(\phi^\epsilon_{+,-}(\xi,\eta))\cdot\varphi_{\leq-K}(\f{|\xi-\eta|}{|\eta|})\\
&\qquad\cdot\varphi_{k}(|\xi|)
\varphi_{k_1}(|\xi-\eta|)\varphi_{k_2}(|\eta|)d\eta d\xi dt\bigr|\\
&\lesssim \sqrt\epsilon2^{-\f{K}{2}}t\sup_{\tau\in[0,t]}\bigl(\|\na f_{k_1}\|_{L^2}\|g_k\|_{L^2}\|g_{k_2}\|_{L^2}\bigr),
\end{aligned}\eeno
which along with \eqref{ansatz 2} and \eqref{equivalent energy 2} implies
\beq\label{estimate for Q + in med 1}\begin{aligned}
&\bigl|\int_0^t\int_{\R^2\times\R^2}\mathfrak{Q}^\epsilon_{+}(\xi,\eta)\varphi_{[-1,5]}(\sqrt\epsilon|\eta|)\cdot
\varphi_{\leq-D-1}(\phi^\epsilon_{+,-}(\xi,\eta))\cdot\varphi_{\leq-K}\bigl(\f{|\xi-\eta|}{|\eta|}\bigr)d\eta d\xi dt\bigr|\\
&\lesssim \sqrt\epsilon2^{-\f{K}{2}}t\sup_{\tau\in[0,t]}\bigl(\|\na f\|_{L^2}\|g\|_{L^2}^2\bigr)\lesssim\sqrt\epsilon2^{-\f{K}{2}}t.
\end{aligned}\eeq

{\it (ii) For case $\f{|\xi-\eta|}{|\eta|}\in[2^{-K},2^{-5}]$,} using \eqref{bound of tilde q eps}, we have
\beq\label{B16}
|\tilde{q}^\epsilon_{+,-}(\xi,\eta)|\lesssim\epsilon 2^K|\xi-\eta|.
\eeq
We denote the integral regime in this case by
\beq\label{integral regime}
\mathbb{S}^\epsilon_{+}=\{(\xi,\eta)\in\text{supp}\,(q^\epsilon_{+,-})\,|\,|\phi^\epsilon_{+,-}(\xi,\eta)|\leq2^{-D},\quad
\f{|\xi-\eta|}{|\eta|}\in[2^{-K},2^{-5}]\}.
\eeq

Firstly, transforming the variables $(\xi,\eta)$ in polar variables $(r_\xi,\theta_\xi,r_\eta,\theta_\eta)$, using \eqref{B16}, we have
\beq\label{B11}\begin{aligned}
&\bigl|\int_0^t\int_{\mathbb{S}^\epsilon_{+}}\mathfrak{Q}^\epsilon_{+}(\xi,\eta)\varphi_{[-1,5]}(\sqrt\epsilon|\eta|)\cdot
\varphi_{\leq-D-1}(\phi^\epsilon_{+,-}(\xi,\eta))\cdot\varphi_{\geq-K+1}\bigl(\f{|\xi-\eta|}{|\eta|}\bigr)d\eta d\xi dt\bigr|\\
&\lesssim\epsilon 2^K t\sup_{\tau\in[0,t]}\int_{\mathbb{S}^\epsilon_{+}}
|\xi-\eta||\widehat{f^+}(\tau,\xi-\eta)|\cdot
|\widehat{g^-}(\tau,\eta)|\cdot|\widehat{g^-}(\tau,-\xi)|\cdot r_\xi r_\eta d\theta_\eta dr_\eta d\theta_\xi dr_\xi.
\end{aligned}\eeq
Here and in what follows, we use an abuse of  notations for the functions in different coordinates.

In order to use the volume of $\mathbb{S}^\epsilon_{+}$, we introduce the coordinates transformation
on $\mathbb{S}^\epsilon_{+}$ as follows:
\beno\begin{aligned}
\Psi_+:\,&\mathbb{S}^\epsilon_{+}\rightarrow\widetilde{\mathbb{S}}^\epsilon_{+}\subset\R^2\times\R^2,\\
&(r_\xi,\theta_\xi,r_\eta,\theta_\eta)\mapsto(r_\xi,\theta_\xi,\tilde{r}_\eta,\theta_\eta)
=(r_\xi,\theta_\xi,\phi^\epsilon_{+,-}(\xi,\eta),\theta_\eta),
\end{aligned}\eeno
we have
\beno
\det\Bigl(\f{\p\Psi_+(r_\xi,\theta_\xi,r_\eta,\theta_\eta)}{\p(r_\xi,\theta_\xi,r_\eta,\theta_\eta)}\Bigr)
=\f{\p\phi^\epsilon_{+,-}(\xi,\eta)}{\p r_\eta}.
\eeno
Thanks to the expression of $\phi^\epsilon_{+,-}(\xi,\eta)$ in \eqref{phi + -},
we have
\beno\begin{aligned}
&\f{\p\phi^\epsilon_{+,-}(\xi,\eta)}{\p r_\eta}=4\epsilon\cos^2\bigl(\f12\angle(\xi,\eta)\bigr)\bigl(2r_\eta-r_\xi\bigr)\\
&\qquad+\epsilon\bigl(4\cos^2\bigl(\f12\angle(\xi,\eta)\bigr)-3\bigr)\bigl[(r_\xi+r_\eta)\p_{r_\eta}|\xi-\eta|+\p_{r_\eta}|\xi-\eta|^2+|\xi-\eta|\bigr].
\end{aligned}\eeno

Since
\beno
|\xi-\eta|^2=r_\xi^2+r_\eta^2-2r_\xi r_\eta\cos\bigl(\angle(\xi,\eta)\bigr),
\eeno
we have
\beno
\p_{r_\eta}|\xi-\eta|^2=2\bigl(r_\eta-r_\xi\cos\bigl(\angle(\xi,\eta)\bigr)\bigr),
\quad\p_{r_\eta}|\xi-\eta|=\f{1}{|\xi-\eta|}\bigl(r_\eta-r_\xi\cos\bigl(\angle(\xi,\eta)\bigr)\bigr).
\eeno
Without loss of generality, we only consider $\angle(\xi,\eta)\neq0$. Since
\beno
r_\eta-r_\xi\cos\angle(\xi,\eta)=-|\xi-\eta|\cos\angle(\xi-\eta,\eta),
\eeno
we have
\beq\label{B17}
\p_{r_\eta}|\xi-\eta|^2=-2|\xi-\eta|\cos\angle(\xi-\eta,\eta),\quad\p_{r_\eta}|\xi-\eta|=-\cos\angle(\xi-\eta,\eta),
\eeq
which implies that
\beno\begin{aligned}
&\f{\p\phi^\epsilon_{+,-}(\xi,\eta)}{\p r_\eta}=4\epsilon\cos^2\bigl(\f12\angle(\xi,\eta)\bigr)\bigl(2r_\eta-r_\xi\bigr)\\
&\qquad
-\epsilon\bigl(4\cos^2\bigl(\f12\angle(\xi,\eta)\bigr)-3\bigr)\bigl[\cos\bigl(\angle(\xi-\eta,\eta)\bigr)(r_\xi+r_\eta+2|\xi-\eta|)
-|\xi-\eta|\bigr].
\end{aligned}\eeno

Using \eqref{angle between xi and eta}, the facts that $|\xi-\eta|\leq\f{1}{32}r_\eta$ and $r_\xi\approx r_\eta$, we have
\beq\label{B9}
\det\Bigl(\f{\p\Psi_+(r_\xi,\theta_\xi,r_\eta,\theta_\eta)}{\p(r_\xi,\theta_\xi,r_\eta,\theta_\eta)}\Bigr)
=\f{\p\phi^\epsilon_{+,-}(\xi,\eta)}{\p r_\eta}\approx4\epsilon r_\eta
-2\epsilon r_\eta\cos\bigl(\angle(\xi-\eta,\eta)\bigr)\sim\epsilon r_\eta,
\eeq
which along with the fact that $\sqrt\epsilon r_\eta\sim1$ yields
\beq\label{det of Jacobi}
\det\Bigl(\f{\p\Psi_+(r_\xi,\theta_\xi,r_\eta,\theta_\eta)}{\p(r_\xi,\theta_\xi,r_\eta,\theta_\eta)}\Bigr)\sim\sqrt\epsilon.
\eeq

Changing the variables $(r_\xi,\theta_\xi,r_\eta,\theta_\eta)$ to $(r_\xi,\theta_\xi,\tilde{r}_\eta,\theta_\eta)$, using \eqref{det of Jacobi}, we deduce from \eqref{B11} that
\beno\begin{aligned}
&\bigl|\int_0^t\int_{\mathbb{S}^\epsilon_{+}}\mathfrak{Q}^\epsilon_{+}(\xi,\eta)\varphi_{[-1,5]}(\sqrt\epsilon|\eta|)
\varphi_{\leq-D-1}(\phi^\epsilon_{+,-}(\xi,\eta))\cdot\varphi_{\geq-K+1}\bigl(\f{|\xi-\eta|}{|\eta|}\bigr)d\eta d\xi dt\bigr|\\
&\lesssim \sqrt\epsilon 2^Kt\sup_{\tau\in[0,t]}\int_{\f{31}{128\sqrt\epsilon}}^{\f{66}{\sqrt\epsilon}}\int_0^{2\pi}
\int_0^{2^{-D}}\int_0^{2\pi}
|\xi-\eta||\widehat{f^+}(\tau,\xi-\eta)|\\
&\quad\cdot
|\widehat{g^-}(\tau,\eta)|\cdot|\widehat{g^-}(\tau,-\xi)|\cdot 1_{\mathbb{S}^\epsilon_{+}}(\xi,\eta)\cdot r_\xi r_\eta d\theta_\eta d{\tilde{r}_\eta} d\theta_\xi dr_\xi\\
&\lesssim \sqrt\epsilon 2^Kt\sup_{\tau\in[0,t]}\Bigl\{\|g(\tau,\cdot)\|_{L^2}\cdot\Bigl(\int_{\f{31}{128\sqrt\epsilon}}^{\f{66}{\sqrt\epsilon}}\int_0^{2\pi}
\Bigl(\int_0^{2^{-D}}\int_0^{2\pi}
|\xi-\eta||\widehat{f^+}(\tau,\xi-\eta)|\\
&\quad\cdot
|\widehat{g^-}(\tau,\eta)|\cdot 1_{\mathbb{S}^\epsilon_{+}}(\xi,\eta)\cdot  r_\eta d\theta_\eta d{\tilde{r}_\eta}\Bigl)^2\cdot r_\xi d\theta_\xi dr_\xi\Bigr)^{\f12}\Bigr\}\\
&\lesssim \sqrt\epsilon 2^K 2^{-\f{D}{2}}t\sup_{\tau\in[0,t]}\Bigl\{\|g(\tau,\cdot)\|_{L^2}\cdot\Bigl(\int_{\f{31}{128\sqrt\epsilon}}^{\f{66}{\sqrt\epsilon}}\int_0^{2\pi}
\int_0^{2^{-D}}\int_0^{2\pi}
|\xi-\eta|^2|\widehat{f^+}(\tau,\xi-\eta)|^2\\
&\quad\cdot
|\widehat{g^-}(\tau,\eta)|^2\cdot 1_{\mathbb{S}^\epsilon_{+}}(\xi,\eta)\cdot  r_\eta^2 r_\xi d\theta_\eta d{\tilde{r}_\eta} d\theta_\xi dr_\xi\Bigr)^{\f12}\Bigr\}.
\end{aligned}\eeno
Changing variables $(r_\xi,\theta_\xi,\tilde{r}_\eta,\theta_\eta)$
to $(r_\xi,\theta_\xi,r_\eta,\theta_\eta)$, using \eqref{det of Jacobi} and the fact that $r_\eta=|\eta|\leq 2^K|\xi-\eta|$, we have
\beno\begin{aligned}
&\bigl|\int_0^t\int_{\mathbb{S}^\epsilon_{+}}\mathfrak{Q}^\epsilon_{+}(\xi,\eta)\varphi_{[-1,5]}(\sqrt\epsilon|\eta|)
\varphi_{\leq-D-1}(\phi^\epsilon_{+,-}(\xi,\eta))d\eta d\xi dt\bigr|\\
&\lesssim\epsilon^{\f34} 2^{\f{3K}{2}}2^{-\f{D}{2}}t\sup_{\tau\in[0,t]}\Bigl\{\|g(\tau,\cdot)\|_{L^2}\Bigl(\int_{\mathbb{S}^\epsilon_{+}}
|\xi-\eta|^3|\widehat{f^+}(\tau,\xi-\eta)|^2\cdot
|\widehat{g^-}(\tau,\eta)|^2\cdot r_\xi r_\eta d\theta_\eta dr_\eta d\theta_\xi dr_\xi\Bigr)^{\f12}\Bigl\}\\
&\lesssim\epsilon^{\f34} 2^{\f{3K}{2}}2^{-\f{D}{2}}t\sup_{\tau\in[0,t]}\bigl(\|f\|_{H^{\f32}}\|g\|_{L^2}^2\bigr),
\end{aligned}\eeno
which along with \eqref{ansatz 2} and \eqref{equivalent energy 2} implies that
\beq\label{estimate for Q + in med 2}\begin{aligned}
&\bigl|\int_0^t\int_{\mathbb{S}^\epsilon_{+}}\mathfrak{Q}^\epsilon_{+}(\xi,\eta)\varphi_{[-1,5]}(\sqrt\epsilon|\eta|)\cdot
\varphi_{\leq-D-1}(\phi^\epsilon_{+,-}(\xi,\eta))\cdot\varphi_{\geq-K+1}\bigl(\f{|\xi-\eta|}{|\eta|}\bigr)d\eta d\xi dt\bigr|\\
&\lesssim\epsilon^{\f34} 2^{\f{3K}{2}}2^{-\f{D}{2}}t.
\end{aligned}\eeq

Due to \eqref{estimate for Q + in med 1} and \eqref{estimate for Q + in med 2}, taking
\beno
\sqrt\epsilon2^{-\f{K}{2}}\sim\epsilon^{\f34} 2^{\f{3K}{2}}2^{-\f{D}{2}},\quad\text{i.e.,  } 2^K\sim2^{\f{D}{4}}\epsilon^{-\f18},
\eeno
we have
\beq\label{estimate for Q + in med freq with small modulation}
\bigl|\int_0^t\int_{\mathbb{S}^\epsilon_{+}}\mathfrak{Q}^\epsilon_{+}(\xi,\eta)\varphi_{[-1,5]}(\sqrt\epsilon|\eta|)
\varphi_{\leq-D-1}(\phi^\epsilon_{+,-}(\xi,\eta))d\eta d\xi dt\bigr|\lesssim2^{-\f{D}{8}}\epsilon^{\f{9}{16}}t.
\eeq

\medskip

Thanks to \eqref{estimate for Q + in med freq with large modulation} and \eqref{estimate for Q + in med freq with small modulation}, taking
\beno
2^D\epsilon^{\f32}\sim2^{-\f{D}{8}}\epsilon^{\f{9}{16}},\quad\text{i.e.,  } 2^D\sim\epsilon^{-\f56},
\eeno
we have
\beno
\bigl|\int_0^t\int_{\R^2\times\R^2}\mathfrak{Q}^\epsilon_{+}(\xi,\eta)\varphi_{[-1,5]}(\sqrt\epsilon|\eta|)d\eta d\xi dt\bigr|\lesssim1+\epsilon^{\f23}t,
\eeno
which along with \eqref{estimate for Q + in low freq} implies
\beq\label{estimate for Q +}
\bigl|\int_0^t\int_{\R^2\times\R^2}\mathfrak{Q}^\epsilon_{+}(\xi,\eta)d\eta d\xi dt\bigr|\lesssim1+\epsilon^{\f23}t.
\eeq

{\it Step 3.4. Estimate for $\int_0^t\int_{\R^2\times\R^2}\mathfrak{Q}^\epsilon_{-}(\xi,\eta)d\eta d\xi d\tau$.} Thanks to \eqref{phase - -} and \eqref{bound of tilde q eps}, we could derive the estimate for $\int_0^t\int_{\R^2\times\R^2}\mathfrak{Q}^\epsilon_{-}(\xi,\eta)d\eta d\xi d\tau$ in a similar way as that for $\int_0^t\int_{\R^2\times\R^2}\mathfrak{Q}^\epsilon_{+}(\xi,\eta)d\eta d\xi d\tau$. That is
\beq\label{estimate for Q -}
\bigl|\int_0^t\int_{\R^2\times\R^2}\mathfrak{Q}^\epsilon_{-}(\xi,\eta)d\eta d\xi dt\bigr|\lesssim1+\epsilon^{\f23}t.
\eeq
To archive \eqref{estimate for Q -}, we only need  to check that on
\beno
\mathbb{S}^\epsilon_{-}=\{(\xi,\eta)\in\text{supp}\,(q^\epsilon_{-,-})\,|\,|\phi^\epsilon_{-,-}(\xi,\eta)|\leq2^{-D},\quad
\f{|\xi-\eta|}{|\eta|}\in[2^{-K},2^{-5}]\},
\eeno
there exists a coordinates transformation $\Psi_-$ such that
\beq\label{det of Jacobi -}
\det\Bigl(\f{\p\Psi_-(r_\xi,\theta_\xi,r_\eta,\theta_\eta)}{\p(r_\xi,\theta_\xi,r_\eta,\theta_\eta)}\Bigr)\sim\sqrt\epsilon.
\eeq
Indeed,  introducing  the coordinates transformation  $\Psi_-$ as follows
\beno\begin{aligned}
\Psi_-:\,&\mathbb{S}^\epsilon_{-}\rightarrow\widetilde{\mathbb{S}}^\epsilon_{-}\subset\R^2\times\R^2,\\
&(r_\xi,\theta_\xi,r_\eta,\theta_\eta)\mapsto(r_\xi,\theta_\xi,\tilde{r}_\eta,\theta_\eta)
=(r_\xi,\theta_\xi,\phi^\epsilon_{-,-}(\xi,\eta),\theta_\eta),
\end{aligned}\eeno
we have
\beno
\det\Bigl(\f{\p\Psi_-(r_\xi,\theta_\xi,r_\eta,\theta_\eta)}{\p(r_\xi,\theta_\xi,r_\eta,\theta_\eta)}\Bigr)
=\p_{r_\eta}\phi_{-,-}^\epsilon(\xi,\eta).
\eeno

Using \eqref{phi - -}, we have
\beno\begin{aligned}
\p_{r_\eta}\phi_{-,-}^\epsilon(\xi,\eta)&=\epsilon\bigl(2r_\eta-r_\xi\bigr)+\epsilon\bigl(\f{3}{4\cos^2\bigl(\f12\angle(\xi,\eta)\bigr)}-1\bigr)
\bigl[(r_\xi+r_\eta)\p_{r_\eta}|\xi-\eta|+|\xi-\eta|-\p_{r_\eta}|\xi-\eta|^2\bigr],
\end{aligned}\eeno
which along with \eqref{B17} implies
\beno\begin{aligned}
\p_{r_\eta}\phi_{-,-}^\epsilon(\xi,\eta)&=\epsilon\bigl(2r_\eta-r_\xi\bigr)+\epsilon\bigl(1-\f{3}{4\cos^2\bigl(\f12\angle(\xi,\eta)\bigr)}\bigr)\\
&\qquad\times
\bigl[(r_\xi+r_\eta-2|\xi-\eta|)\cos\bigl(\angle(\xi-\eta,\eta)\bigr)-
|\xi-\eta|\bigr].
\end{aligned}\eeno

By virtue of \eqref{angle between xi and eta}, using the facts that $|\xi-\eta|\leq\f{1}{32}|\eta|$ and $|\xi|\approx|\eta|$, we have
\beno
\p_{r_\eta}\phi_{-,-}^\epsilon(\xi,\eta)\approx\epsilon r_\eta \bigl[1+\f12\cos\bigl(\angle(\xi-\eta,\eta)\bigr)\bigr]\sim\epsilon r_\eta,
\eeno
which along with the fact that $\sqrt\epsilon |\eta|\sim1$ implies
\beno\label{B18}
\det\Bigl(\f{\p\Psi_-(r_\xi,\theta_\xi,r_\eta,\theta_\eta)}{\p(r_\xi,\theta_\xi,r_\eta,\theta_\eta)}\Bigr)
=\p_{r_\eta}\phi_{-,-}^\epsilon(\xi,\eta)\sim\sqrt\epsilon,\quad\text{for any}\,\,(\xi,\eta)\in\mathbb{S}^\epsilon_{-}.
\eeno
This is exactly \eqref{det of Jacobi -}.

{\it Step 3.5. Estimate for $III$.} Thanks to \eqref{estimate for Q +} and \eqref{estimate for Q -}, we obtain \eqref{estimate for III}.

\medskip

{\bf Step 4. The a priori energy estimate.}  Combining the estimates \eqref{priori for Bsq}, \eqref{estimate for I and II}, \eqref{estimate for III}, we arrive at the final energy estimate \eqref{priori energy}. This completes the proof of the proposition.
\end{proof}

\begin{remark}
A variant of the long time issue considered in this paper would be to look for the lifespan of solutions of
\beq\label{Bsq2Dbis}\left\{\begin{aligned}
&\p_t\z+(1+\D)\na\cdot\vv v+\na\cdot(\z\vv v)=0,\\
&\p_t\vv v+(1+\D)\na\z+\f{1}{2}\na\bigl(|\vv v|^2\bigr)=\vv 0,
\end{aligned}\right.\eeq
with the initial data
\beq\label{initial a}
\z|_{t=0}=\z_0=O(\epsilon),\quad \vv v|_{t=0}=\vv v_0=O(\epsilon).
\eeq

This issue was studied in \cite{SX2} in the one-dimensional case and the lifespan was proven to be $O(1/\epsilon^{4/3}),$ improving the $O(1/\epsilon)$ result obtains par pure dispersive methods. By adapting the method in \cite{SX2} to the two-dimensional case one can obtain the same  $O(1/\epsilon^{4/3})$ result for \eqref{Bsq2Dbis}.

\end{remark}

\vspace{0.5cm}
\noindent {\bf Acknowledgments.}  The work of the first  author was partially  supported by the ANR project ANuI (ANR-17-CE40-0035-02).
The work
of the second author was partially supported by NSF of China under grants 11671383.

\end{document}